\documentclass[12pt, leqno]{amsart}
\usepackage{amssymb,url}
\usepackage[mathscr]{eucal}
\usepackage{mathrsfs}
\usepackage{mystyle}
\usepackage{tikz}
\usepackage{enumitem}
\usepackage{multicol}
\usepackage{listings}
\usepackage{xcolor}
\usepackage{subcaption} % Add the subcaption package

\lstdefinestyle{mystyle}{
    language=Mathematica,
    backgroundcolor=\color{lightgray},
    keywordstyle=\ttfamily,
    basicstyle=\ttfamily\small,
    breakatwhitespace=false,         
    breaklines=true,                 
    captionpos=b,                    
    keepspaces=true,                 
    numbers=none,                    % Set this to none to disable line numbers
    showspaces=false,                
    showstringspaces=false,
    showtabs=false,                  
    tabsize=2
}

\begin{document}
\title{Spectrum of the Laplacian on the Fricke-Macbeath surface}
\author{Chul-hee Lee}
\address{June E Huh Center for Mathematical Challenges, Korea Institute for Advanced Study, Seoul 02455, Korea}
\email{chlee@kias.re.kr}
%\keywords{}
%\subjclass[2010]{81R50, 16G, 16T25,17B37}
\date{\today}
%\thanks{This work was supported by ..}
\begin{abstract}
The Fricke-Macbeath surface is the unique Hurwitz surface of genus 7 with 504 conformal automorphisms.
In this paper, we prove that the first eigenvalue of the Laplacian on the Fricke-Macbeath surface has a sevenfold multiplicity and contained in the interval $[1.23, 1.26]$.
Further, we numerically identify the 7-dimensional representation of its automorphism group corresponding to the eigenspace associated to the first eigenvalue.
We also determine the Dirichlet domain centered at 0 for a Fuchsian group that uniformizes the Fricke-Macbeath surface, identifying the algebraic coordinates for its vertices.
\end{abstract}
\maketitle
\section{Introduction}
In 1893, Hurwitz proved that a compact Riemann surface \(X\) of genus \(g \geq 2\) has at most \(84(g - 1)\) conformal automorphisms~\cite{hurwitz1893}.
This upper bound is now known as the Hurwitz bound, and a Riemann surface attaining this bound is called a Hurwitz surface.

Notably, for Riemann surfaces of low genera, those meeting the Hurwitz bound are rare.
One prominent example is the Klein quartic~\cite{Klein79}, described by the equation
\[
x_0^3x_1 + x_1^3x_2 + x_2^3x_0 = 0,
\]
in complex projective space.
The Klein quartic predates Hurwitz's seminal work and stands as the first of its kind to have ever been discovered.
This surface possesses 168 conformal automorphisms, with its automorphism group being isomorphic to the simple group $\PSL_2(7)$.
For a comprehensive exploration of the history and significance of Klein's quartic, as well as the subsequent advancements inspired by it, one can refer to~\cite{MR1722410}.

In 1899, Fricke made a significant contribution to the subject by unveiling another Hurwitz surface, particularly one of genus 7~\cite{MR1511059}.
This surface possesses the automorphism group $\PSL_2(8)$, a simple group of order 504.
Subsequently, in 1965, without being aware of Fricke's earlier work, Macbeath determined its algebraic model~\cite{macbeath1965}. 
This surface is now known as the Fricke-Macbeath surface, and alongside the Klein quartic, it holds the distinction of being one of the only two instances of Hurwitz surfaces with genera less than 14.
Interestingly, in the case of genus 14, there exist three distinct Hurwitz surfaces with the same automorphism group of order 1092.

The abundant symmetries inherent in Hurwitz surfaces facilitates intriguing interactions with various structures associated with a surface, making their study compelling.
This interplay makes Hurwitz surfaces a fertile ground for researchers across diverse fields.
Notably, in the realm of arithmetic, some of these surfaces hold significant recognition as Shimura curves, a fact highlighted by Shimura~\cite{MR0204426}.
Their connection with Galois theory is also a subject of discussion in Serre's enlightening letter to Abhyankar~\cite{MR1272022}.
For their role in the theory of dessins d'enfants, see the book by Jones and Wolfart~\cite{MR3467692}.

In this paper, we explore the interaction between symmetries and the spectral properties of such a surface.
According to the uniformization theorem, $X$ is conformally equivalent to $\Gamma\backslash \uhp$, where
$\uhp$ is the upper half-plane in the complex plane and $\Gamma$ is discrete, torsion-free subgroup of $\PSL_2(\R)$.
Then $\Aut(X) \cong N(\Gamma) / \Gamma$, where $N$ is the normalizer of $\Gamma$ in $\PSL_{2}(\R)$. 
In particular, $X$ is a Hurwitz surface when $N(\Gamma)$ is the $(2,3,7)$-triangle group.
Now $X$ admits a hyperbolic metric, namely, the Riemannian metric of constant Gaussian curvature -1.
The elements of $\Aut(X)$ act as orientation-preserving isometries of $X=\Gamma\backslash \uhp$.
The Laplacian $\Delta$ for $X$ has non-negative discrete eigenvalues
\[
0=\la_0< \la_1 \leq \la_2 \leq \ldots \to \infty.
\]
An interesting feature of a Riemann surface with many automorphisms is that each eigenspace of the Laplacian can exhibit high multiplicities, as each of these eigenspaces of the Laplacian serves as a representation of $\Aut(X)$.

There are only a few known cases of Riemann surfaces with non-trivial automorphism groups for which we have established values for the multiplicities $m_1$ of $\la_1$.
To the best of our knowledge, the Bolza surface and the Klein quartic are the only known cases.
The Bolza surface, despite not meeting the Hurwitz bound, is a genus 2 Riemann surface renowned for having the highest number of conformal symmetries, totaling 48, within its genus.
In a study by Jenni~\cite{jenni1984}, it was asserted that $m_1=3$ for the Bolza surface; however, Cook~\cite{cook2018} identified a gap in its proof.
Cook subsequently narrowed down the possible range to $3 \leq m_1 \leq 4$.
A subsequent investigation conducted in \cite{FBP21} successfully confirmed that the correct value for $m_1$ is indeed 3.
Concerning the Klein quartic, Cook's analysis indicated that the multiplicity of the first eigenvalue lies within the range of 6 to 8.
A rigorous proof establishing $m_1=8$ is also provided in~\cite{FBP21}.

Surfaces with many automorphisms like the Bolza surface and the Klein quartic often emerge as solutions to extremal problems concerning both $m_1$ and $\la_1$.
In ~\cite{FBP21}, it was shown that the Klein quartic attains the maximal $m_1$ among all hyperbolic surfaces of genus 3.
In a work by Strohmaier and Uski \cite{MR3009726}, they conjectured that the Bolza surface maximizes $\la_1$ among all hyperbolic surfaces of genus 2 based on numerical experiments.
A recent study \cite{kravchuk2023automorphic} provides additional evidence supporting the conjecture that both the Bolza surface and the Klein quartic maximize $\la_1$ among hyperbolic surfaces of genus 2 and 3, respectively.

Our main result is as follows:
\begin{thmnn}
The multiplicity of the first eigenvalue $\la_1$ of the Fricke-Macbeath surface is 7, and $\la_1\in [1.23, 1.26]$.
\end{thmnn}
Our primary tool for the proof is the Selberg trace formula.
We make use of Vogeler's results~\cite{Vogeler2003} concerning the length spectrum of the Fricke-Macbeath surface, specifically, we require information about the conjugacy classes of the two shortest lengths in the spectrum.
When applying the Selberg trace formula, we utilize various types of test functions. 
These include the functions featured in~\cite[Section 4]{MR2338122} and~\cite[Section 5]{MR4186122}, which played a crucial role in confirming the Selberg eigenvalue conjecture for numerous congruence subgroups with small conductor.
We also employ the test functions introduced in the context of the sphere packing problem~\cite[Section 7]{MR1973059}, which were recently employed in the investigation of spectral properties of hyperbolic surfaces in~\cite{FBP21}.

After establishing the multiplicity of an eigenvalue, the subsequent natural question is to identify the associated eigenspace as a representation.
To address this, we present a numerical method for computing the character of a representation obtained from a set of eigenfunctions.
This method enables us to formulate a conjecture regarding which 7-dimensional representation of $\PSL_2(8)$ corresponds to the representation associated with $\la_1$ of the Fricke-Macbeath surface; see Conjecture~\ref{conj:rep7}.
The application of these algorithms to compute the characters for other Riemann surfaces with many automorphisms will be discussed in a separate work.

To conduct a numerical calculation based on the finite element method, we determine a fundamental domain for the Fuchsian group that uniformizes the Fricke-Macbeath surface using Poincar\'{e}'s polyhedron theorem~\cite{maskit1971}, along with its side-pairing transformations.
It is a hyperbolic polygon with 42 sides, which is actually the Dirichlet domain centered at 0.
We offer a visualization of this fundamental domain, updating the classic Fricke-Klein illustration~\cite{MR0352287}; see Figures \ref{fig:fricke_tess1} and \ref{fig:fricke_tess2} for a comparison.
Unlike the case of the Klein quartic in Figure~\ref{fig:237_klein}, we find that this polygon is not completely tessellated by the fundamental $(2,3,7)$-triangles.
We note that the fundamental domain discussed in ~\cite{MR1511059} differs from ours.

\begin{figure}
    \begin{minipage}{0.5\textwidth}
        \centering
        \includegraphics[width=\linewidth]{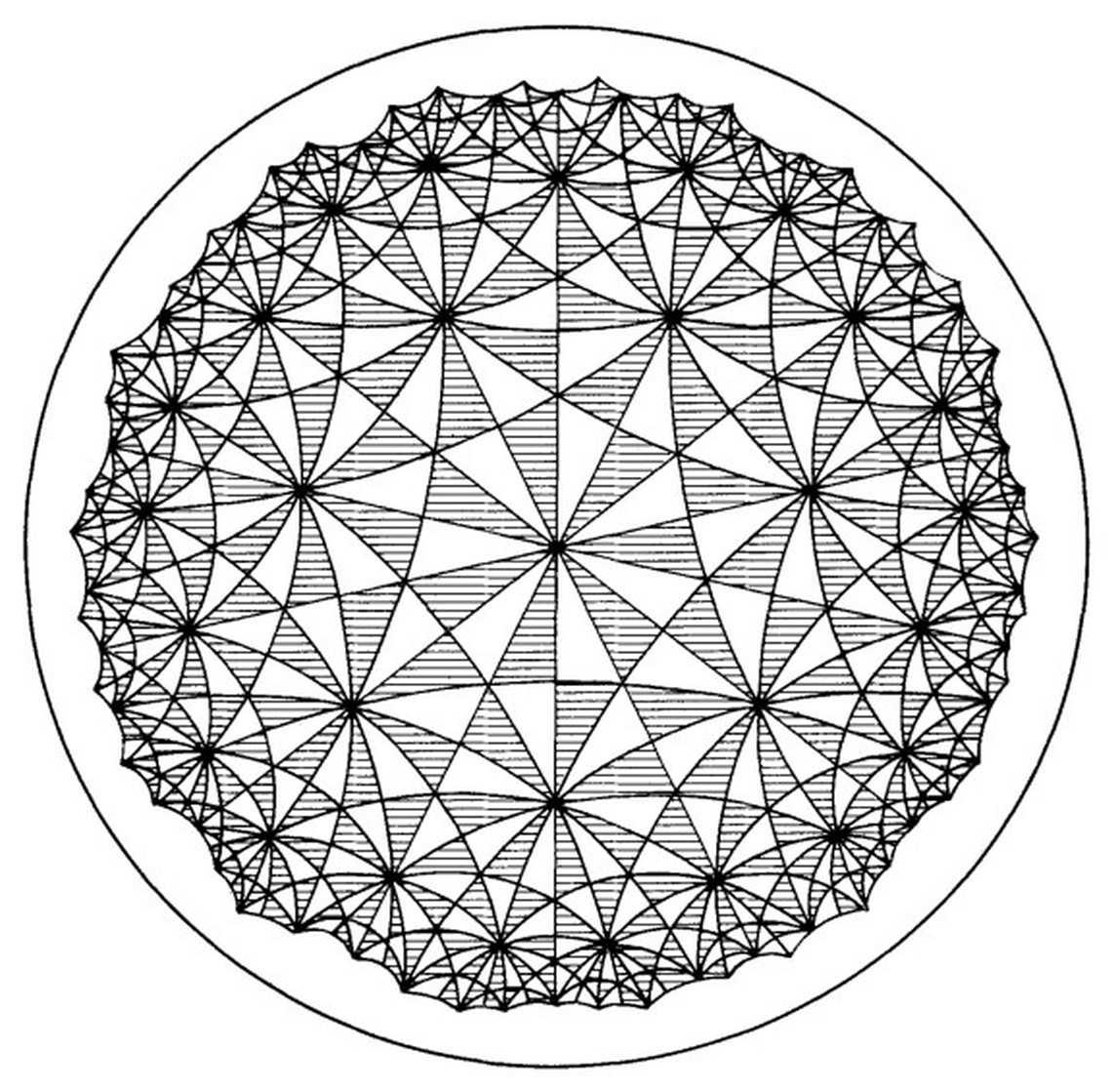}
        \caption{Tessellation of the unit disk by (2,3,7)-triangles \cite[p. 109]{klein1890vorlesungen}}
        \label{fig:fricke_tess1}
    \end{minipage}%
    \begin{minipage}{0.5\textwidth}
        \centering
        \includegraphics[width=\linewidth]{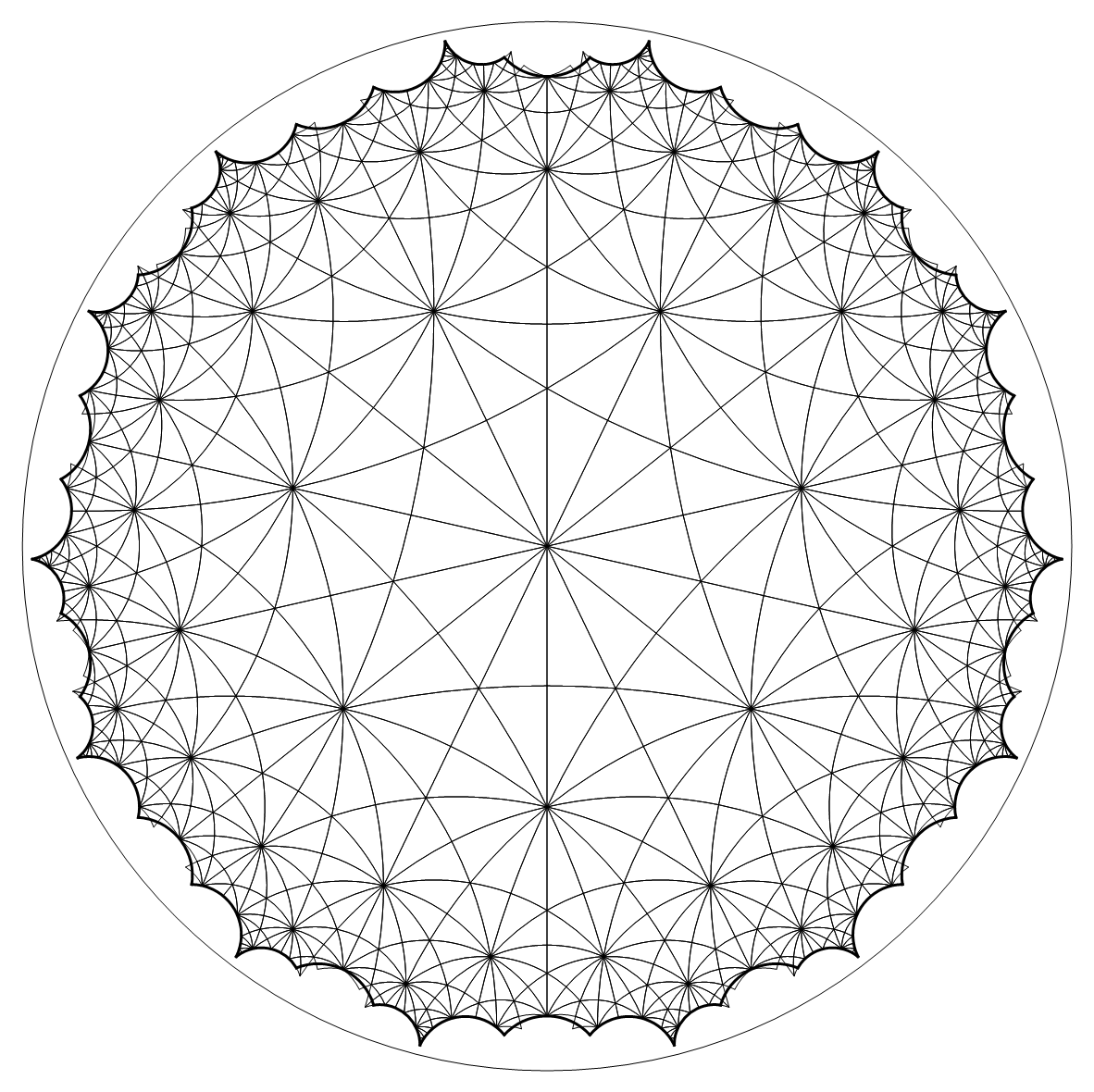}
        \caption{Fundamental domain for the Fricke-Macbeath surface (boundary in thick lines)}
        \label{fig:fricke_tess2}
    \end{minipage}
\end{figure}

The paper is organized as follows.
In Section 2, we delve into the background, discussing topics like the Laplacian on a Riemannian manifold, the hyperbolic plane, and the Selberg trace formula.
Section 3 focuses on presenting an explicit uniformization of the Fricke-Macbeath surface and the identification of its fundamental domain.
In Section 4, we prove our main result on the first eigenvalue of the Fricke-Macbeath surface.
Lastly, in Section 5, we present a numerical method to determine the character of a representation obtained from a set of eigenfunctions, and propose a conjecture identifying the 7-dimensional representation of $\PSL_2(8)$ associated to the first eigenvalue of the Fricke-Macbeath surface.

The accompanying computer code for this paper is available at \url{https://github.com/chlee-0/fricke_macbeath}.

\section{Background}
\subsection{Laplacian on Riemannian manifolds}
In this subsection, we provide a brief review of the properties of the Laplacian on a Riemannian manifold.
For a detailed discussion on this topic, see~\cite{MR2742784}.
Consider a Riemannian manifold $M$ with the metric given by
\[ds^2 = \sum_{j,k} g_{jk} dx^j dx^k.\]
The Laplacian \(\Delta\), acting on \( C^\infty(M) \), can be locally expressed by:
\[ 
\Delta = -\sum_{j,k} \frac{1}{\sqrt{\det g}} \frac{\partial}{\partial x^j}
\left( \sqrt{\det g}\, g^{jk} \frac{\partial}{\partial x^k} \right),
\]
where \( g^{jk} \) are the entries of the inverse matrix of \( (g_{jk}) \) and \( \det g \) stands for the determinant of the matrix \( (g_{jk}) \).
For a compact connected manifold, $\Delta$ has non-negative discrete eigenvalues
\[
0=\la_0< \la_1 \leq \la_2 \leq \ldots \to \infty,
\]
with corresponding \(C^{\infty}\)-eigenfunctions given by 
\[
\Delta \phi_{i}=\la_{i} \phi_{i},
\]
which together provide an orthonormal basis for \(L^{2}(M)\).

For a Riemannian manifold, the Laplacian $\Delta$ commutes with isometries of $M$.
This implies that for an eigenfunction $f$ of $\Delta$ and an isometry $g$, 
$f\circ g^{-1}$ is an eigenfunction with the same eigenvalue as $f$.
Consequently, the complex vector space of eigenfunctions with a given eigenvalue becomes a representation for a group of isometries of $M$.

\subsection{Two models of hyperbolic plane}
In this subsection, we review the two models of the hyperbolic plane. 
One such model is the upper half-plane, denoted by \( \uhp \). This consists of complex numbers \( z \) where the imaginary part is positive: \( \uhp = \{ z \in \C : \Im(z) > 0 \} \).
This model of the hyperbolic plane carries a metric:
\[ ds^2 = \frac{dx^2 + dy^2}{y^2},\]
which yields a constant Gaussian curvature $-1$.
The distance $\distH(z,z')$ between two points $z$ and $z'$ is given by 
\[
\cosh \distH(z,z') = 1 + \frac{|z-z'|^2}{2\Im z \Im z'}.
\]
Geodesics in \( \uhp \) are represented as either vertical half-lines or semicircles that intersect the $x$-axis perpendicularly.
For any two distinct points in \( \uhp \), there is a unique geodesic connecting them. 
The Laplacian is given by
\[ -\Delta = y^2 \left( \frac{\partial^2}{\partial x^2} + \frac{\partial^2}{\partial y^2} \right). \]

The group \(\PSL_2(\R)=\SL_2(\R)/\{\pm 1\}\) acts isometrically on \(\uhp\).
The action is given by \(z \mapsto gz = \frac{az+b}{cz+d}\), where \(g = \pm \pmat abcd \in\PSL_2(\R)\).
A non-trivial element \( g \) of \( \PSL_2(\R) \) is called hyperbolic if
$|\Tr(g)| > 2$.
This is equivalent to 
\( g \) fixing two distinct points on the boundary \( \partial \uhp\) of the upper half-plane.
The \textit{length} \( \ell(g) \) of a hyperbolic element \( g \) is defined as:
\[
\ell(g) = 2 \arcosh\left(\frac{|\Tr(g)|}{2}\right).
\]

The unit disk \( \disk \) provides another model for the hyperbolic plane.
It is defined by \( \disk = \{ z \in \mathbb{C} : |z| < 1 \} \).
In this model, the metric is expressed as:
\[ ds^2 = \frac{4(dx^2 + dy^2)}{(1 - x^2 - y^2)^2}\]
which results in a constant Gaussian curvature $-1$.
An isometry that maps \( \uhp \) onto \( \disk \) is given by:
\[
z \mapsto \frac{z-i}{z+i},
\]
which is also a conformal equivalence.
The distance, denoted as $\distD(z,z')$, between two points $z$ and $z'$ in the disk is
\[
\cosh \distD(z,z') = 1 + \frac{2|z-w|^2}{(1-|z|^2)(1-|w|^2)}.
\]
The Laplacian in this setting reads
\[ -\Delta = \frac{(1 - x^2 - y^2)^2}{4} \left( \frac{\partial^2}{\partial x^2} + \frac{\partial^2}{\partial y^2} \right). \]
The group of orientation-preserving isometries of \(\disk\) is given by the group 
\[ \PSU(1,1) = \left\{\pm \pmat{a}{\conj{b}}{b}{\conj{a}} : a, b \in \mathbb{C},\, \abs{a}^{2}-\abs{b}^{2}=1 \right\}. \]
The notion of hyperbolic elements for $\PSU(1,1)$ and their lengths can be defined in the same way.

\subsection{Selberg trace formula}
Let $X=\Gamma\backslash \uhp$ be a compact Riemann surface of genus $g\geq 2$, where $\Gamma$ is discrete, torsion-free subgroup of $\PSL_2(\R)$.
Hence, every non-trivial element of $\Gamma$ is hyperbolic.

The Selberg trace formula establishes a link between the conjugacy classes of $\Gamma$ and and the eigenvalues of the Laplacian.
Before delving into the formula itself, we present some requisite notation.

Given a positive parameter $\varepsilon>0$, define
\[
\mathscr{S}_{\varepsilon}=\left\{r \in \C : | \Im r \mid<\frac{1}{2}+\varepsilon\right\}.
\]
Suppose that $h: \mathscr{S}_{\varepsilon} \rightarrow \C$ is a holomorphic even function fulfilling the condition:
\begin{equation}
h(r)=O\left(1+|r|^{2}\right)^{-1-\varepsilon} \text { uniformly on } \mathscr{S}_{\varepsilon}.
\label{eq:hdecay}
\end{equation}
The Fourier transform of $h$ is then given by:
\[
g(u)=\what{h}(u)=\frac{1}{\sqrt{2 \pi}} \int_{-\infty}^{+\infty} e^{-i r u} h(r)\, d r
, \quad u \in \R.
\]
The pair $h, g$ is called an \textit{admissible transform pair}.

A non-trivial element $S$ of $\Gamma$ is primitive if it does not admit a representation in the form $S = R^m$ for some $R\in \Gamma$ and $m\geq 2$.
For any hyperbolic element $\gamma\in \Gamma$, there corresponds a unique primitive $\gamma_0$ such that $\gamma=\gamma_0^{m}$ for an integer $m \geq 1$; see~\cite[Lemma 9.2.6]{MR2742784}.
The quantity $\La(\gamma)$ is defined as $\ell\left(\gamma_{0}\right)$.

We now state Selberg's trace formula~\cite[Theorem 9.5.3]{MR2742784}.
Let
\[
r(\la) = \begin{cases}
i \sqrt{\frac{1}{4}-\la} & \text{if } 0 \leq \la < \frac{1}{4} \\
\sqrt{\la - \frac{1}{4}} & \text{if } \la \geq \frac{1}{4}
\end{cases}.
\]
Denote by $r_n$ the quantity $r(\la_n)$.
Let \(\calC(\Gamma)\) be the set of conjugacy classes of hyperbolic elements, and $\{\gamma\}$ be the conjugacy class containing $\gamma \in \Gamma$.
With a given admissible transform pair $h,g$, we have
\begin{equation}
\label{eq:STF1}
\sum_{n=0}^{\infty} h(r_{n})
=
\frac{\area(X)}{4 \pi} \int_{-\infty}^{+\infty} r h(r) \tanh (\pi r)\, dr+
\frac{1}{\sqrt{2\pi}}\sum_{\{\gamma\} \in \calC(\Gamma)} \frac{\La(\gamma)}{2 \sinh(\ell(\gamma)/2)} g(\ell(\gamma)).
\end{equation}
All sums and integrals in the formula are absolutely convergent.

Our chosen convention for the Fourier transform is different from the one presented in~\cite{MR2742784}.
This difference leads to an altered constant before the sum on the right-hand side of equation \eqref{eq:STF1}.

In literature, the Selberg trace formula is often phrased in terms of closed geodesics of $X$.
In fact, there is a bijection between the oriented closed geodesics of \(X = \Gamma\backslash \uhp\) and the conjugacy classes in \( \Gamma - \{\id\} \). 
Under this bijection, the length \( \ell(T) \) of a conjugacy class containing $T$ matches the length of the corresponding closed geodesic $\gamma_{T}$.
For a detailed description of this relationship, see ~\cite[Section 9.2]{MR2742784}.

\section{Explicit uniformization of the Fricke-Macbeath surface and its fundamental domain}
\subsection{Hurwitz surfaces and the Fricke-Macbeath surface}\label{subsec:hurwitzsurfaces}
Hurwitz demonstrated in ~\cite[§7]{hurwitz1893} that a group of order \(84(g - 1)\) can act as the automorphism group of a compact Riemann surface of genus \(g\geq 2\) precisely when it is a quotient of the $(2,3,7)$-triangle group $\Delp$, whose presentation is given by
\[\langle a, b \mid a^{3} = b^{7} = (ab)^{2} = 1 \rangle.\]
This transforms the challenge of finding surfaces with \(84(g - 1)\) automorphisms into a purely group-theoretic problem.
Consider $\Delp$ as a set of isometries acting on the hyperbolic plane, say, $\disk$ and let $\Gamma$ be a proper normal subgroup of it, with a finite index.
Notably, any non-trivial element of such $\Gamma$ is hyperbolic.
To clarify, there are no elements of torsion or parabolic type within $\Gamma$; for further references, see~\cite[p. 531]{macbeath1965} and~\cite[p. 4]{Vogeler2003}.
Consequently, the quotient space $\Gamma\backslash \disk$ is a compact Riemann surface and a Hurwitz surface, with an automorphism group isomorphic to $\Delp/\Gamma$.
For a general review of this topic, see~\cite{MR1041434, MR1722414}.

For the Fricke-Macbeath surface $\X$, the unique Hurwitz surface of genus 7, its automorphism group is isomorphic to the simple group $\PSL_2(8)$, having order 504.
Presentations of the group $PSL_2(8)$ can be found in various references such as~\cite{burnside1899},~\cite[p. 97]{coxeter_moser_1957} and~\cite{MR0174621}.
Given that these presentations realize $PSL_2(8)$ as a quotient of the $(2,3,7)$-triangle group, they provide a way to explicitly obtain $\X$, as a quotient of the hyperbolic plane.
By utilizing the GAP~\cite{GAP4} code provided in~\cite{Vogeler2003},
one can confirm that the group defined by the presentation
\[\langle a, b \mid a^{3} = b^{7} = (ab)^{2} = (a^{-1} b^3 a^{-1} b a^{-1} b^3)^2 = 1\rangle\]
has 504 elements; see also Appendix~\ref{sec:appB}.

From this, we have $\X\cong \Gamma'\backslash\disk$, where $\Gamma'$ is the normal subgroup of $\Delp$ generated by the element $(a^{-1} b^3 a^{-1} b a^{-1} b^3)^2$.

In the following subsections, we will consider a Fuchsian group $\Gamma$ generated by a set of 42 generators.
These particular generators play the role of side-pairing transformations for its fundamental domain, which turns out to be a hyperbolic polygon with 42 sides.

\subsection{Definition of $\Gamma$ as a subgroup of a Coxeter group}
Consider a Coxeter group given by the following presentation
\[
\Del = \langle s_1,s_2,s_3 | s_1^2 = s_2^2 = s_3^2 = (s_1s_2)^7=(s_2s_3)^3=(s_1s_3)^2 \rangle.
\]
Then $\Delp$ becomes the index two subgroup of $\Del$, consisting of elements of even length.

The following relations can be readily verified within $\Del$:
\begin{equation}
\label{eq:coxrelation}
\begin{alignedat}{2}
& s_3s_1 && = s_1s_3, \\
& s_3s_2s_3 && = s_2s_3s_2, \\
& s_3s_2s_1s_3 && = s_2s_3s_2s_1, \\
& s_3s_2s_1s_2s_3s_2 && = s_2s_3s_2s_1s_2s_3, \\
& s_2s_1s_2s_1s_2s_1s_2 && = s_1s_2s_1s_2s_1s_2s_1, \\
& s_3s_2s_1s_2s_1s_3s_2s_1s_2s_1s_2s_1 && = s_2s_3s_2s_1s_2s_1s_3s_2s_1s_2s_1s_2.
\end{alignedat}
\end{equation}
A way to verify these relations is by using the geometric representation~\cite[5.3]{MR1066460}.
It is a faithful representation $\rho : \Del \to \GL(\R^3)$ given by
\begin{align*}
\rho(s_1) =
\begin{pmatrix}
-1 & 2\phi & 0 \\
0 & 1 & 0 \\
0 & 0 & 1 \\
\end{pmatrix}, \quad
\rho(s_2) = 
\begin{pmatrix}
1 & 0 & 0 \\
2\phi & -1 & 1 \\
0 & 0 & 1 \\
\end{pmatrix}, \quad
\rho(s_3) = 
\begin{pmatrix}
1 & 0 & 0 \\
0 & 1 & 0 \\
0 & 1 & -1 \\
\end{pmatrix},
\end{align*}
where $\phi = \cos(\frac{\pi}{7})$.

Using the relations from \eqref{eq:coxrelation}, we can carry out word reductions in $\Del$.
This means we replace every occurrence of a term on the left-hand side term with its right-hand side equivalent in any word formulated with the group's generators.
For an implementation, see the Mathematica code snippet in Appendix~\ref{sec:appA}.
We believe that these are complete rewriting rules to obtain a reduced word in a lexicographically minimal form, though we do not intend to seek its proof here.
These relations are sufficient to verify an equality within $\Del$, necessary for the arguments presented in this paper.
See \cite{MR1261121} for a proof of the existence of finite rewriting rules to solve the word problem for some Coxeter groups.

The elements $\gamma_{i}\in \Del$, for $i=1,\dots,42$, are defined as below.
Specifically, the first six elements are given as:
\begin{align*}
\gamma_1& =s_2s_3s_2s_1s_2s_1s_3s_2s_1s_2s_1s_3s_2s_1s_2s_1s_2s_3s_2s_1s_2s_1s_2s_1s_3s_2s_1s_2s_1s_3s_2s_1s_2s_1s_3s_2s_1s_2,\\
\gamma_2& =s_1s_2s_3s_2s_1s_2s_1s_2s_1s_3s_2s_1s_2s_1s_3s_2s_1s_2s_1s_2s_3s_2s_1s_2s_1s_2s_1s_3s_2s_1s_2s_1s_3s_2s_1s_2,\\
\gamma_3& =(s_1s_2)^3(s_3s_2s_1s_2s_1)^3s_2s_3s_2s_1s_2s_1s_2s_1s_3s_2s_1s_2s_1s_2s_3s_2s_1s_2s_1s_3s_2s_1s_2,\\
\gamma_4& =(s_1s_2)^2s_3s_2s_1s_2s_1s_2s_3s_2s_1s_2s_1s_2s_1s_3s_2s_1s_2s_1s_3s_2s_1s_2s_1s_2s_3s_2s_1s_2s_1s_2s_3s_2(s_1s_2)^2,\\
\gamma_5& =s_1s_3s_2s_1s_2s_1s_3s_2s_1s_2s_1s_2s_3s_2s_1s_2s_1s_2s_1s_3s_2s_1s_2s_1s_3s_2s_1s_2s_1s_2s_3s_2s_1s_2s_1s_2,\\
\gamma_6& =s_2s_3s_2s_1s_2s_1s_2s_3s_2s_1s_2s_1s_2s_1s_3s_2s_1s_2s_1s_3s_2s_1s_2s_1s_2s_3s_2s_1s_2s_1s_2s_1s_3s_2s_1s_2s_1s_2.
\end{align*}
For $7\leq i\leq 42$, the elements are recursively defined by the formula:
\begin{equation}
\gamma_{i} : = s_2s_1 \gamma_{i-6} s_1s_2.
\label{eq:gamma_recursive}
\end{equation}
By setting $\gamma_{i} := \gamma_{i \mod {42}}$ for $i\in \Z$, 
we can further check that \eqref{eq:gamma_recursive} holds for all $i\in \Z$.

The inverses of the elements $\gamma_{i}$ are given by the following relations:
\begin{align}
\gamma_{1} \gamma_{33} &= 1, & \gamma_{2} \gamma_{23} &= 1, & \gamma_{3} \gamma_{13} &= 1, & \gamma_{4} \gamma_{18} &= 1, & \gamma_{5} \gamma_{26} &= 1, & \gamma_{6} \gamma_{34} &= 1, \notag \\
\gamma_{7} \gamma_{39} &= 1, & \gamma_{8} \gamma_{29} &= 1, & \gamma_{9} \gamma_{19} &= 1, & \gamma_{10} \gamma_{24} &= 1, & \gamma_{11} \gamma_{32} &= 1, & \gamma_{12} \gamma_{40} &= 1, \notag \\
\gamma_{14} \gamma_{35} &= 1, & \gamma_{15} \gamma_{25} &= 1, & \gamma_{16} \gamma_{30} &= 1, & \gamma_{17} \gamma_{38} &= 1, & \gamma_{20} \gamma_{41} &= 1, & \gamma_{21} \gamma_{31} &= 1, \notag \\
\gamma_{22} \gamma_{36} &= 1, & \gamma_{27} \gamma_{37} &= 1, & \gamma_{28} \gamma_{42} &= 1. \label{eq:sidepairings}
\end{align}

From now on, we denote by $\Gamma$ the subgroup of $\Del$ generated by the elements $\gamma_i$ for $i$ from 1 to 42.
In the next subsection, we will see that they serve as side-pairing transformations for a certain hyperbolic polygon.

\begin{prop}
Let $\Gamma$ be a subgroup of $\Del$ generated by $\gamma_i$ for $i=1,\cdots, 42$. 
Then $\Gamma$ is a normal subgroup of $\Del$ and $\Delp$.
\end{prop}

\begin{proof}
To establish the normality of \(\Gamma\), it suffices to confirm that for each \(i \in \{1,2,3\}\) and every \(j \in \{1,\dots,42\}\), the term \(s_i\gamma_{j} s_i\) is expressible as a product of the elements \(\gamma_k\).
An examination of \eqref{eq:sidepairings} narrows our focus to the indices:
\[
j \in \{1, 2, \dots, 11, 12, 14, 15, 16, 17, 20, 21, 22, 27, 28\}.
\]

The following elements of $\Del$ equate to the identity, which can be verified by \eqref{eq:coxrelation}:
\begin{center}
\begin{tabular}{ccc}
$s_1\gamma_{1}s_1\gamma_{25}$,  &$s_2\gamma_{1}s_2\gamma_{31}$,  &$s_3\gamma_{1}s_3\gamma_{11}\gamma_{40}$, \\
$s_1\gamma_{2}s_1\gamma_{35}$,  &$s_2\gamma_{2}s_2\gamma_{41}$,  &$s_3\gamma_{2}s_3\gamma_{21}$, \\
$s_1\gamma_{3}s_1\gamma_{3}$,  &$s_2\gamma_{3}s_2\gamma_{9}$,  &$s_3\gamma_{3}s_3\gamma_{21}\gamma_{36}\gamma_{27}$, \\
$s_1\gamma_{4}s_1\gamma_{40}$,  &$s_2\gamma_{4}s_2\gamma_{4}$,  &$s_3\gamma_{4}s_3\gamma_{16}\gamma_{25}$, \\
\end{tabular}
\end{center}
\begin{center}
\begin{tabular}{ccc}
$s_1\gamma_{5}s_1\gamma_{32}$,  &$s_2\gamma_{5}s_2\gamma_{38}$,  &$s_3\gamma_{5}s_3\gamma_{16}$, \\
$s_1\gamma_{6}s_1\gamma_{24}$,  &$s_2\gamma_{6}s_2\gamma_{30}$,  &$s_3\gamma_{6}s_3\gamma_{8}\gamma_{39}$, \\
$s_1\gamma_{7}s_1\gamma_{19}$,  &$s_2\gamma_{7}s_2\gamma_{25}$,  &$s_3\gamma_{7}s_3\gamma_{8}\gamma_{34}$, \\
$s_1\gamma_{8}s_1\gamma_{29}$,  &$s_2\gamma_{8}s_2\gamma_{35}$,  &$s_3\gamma_{8}s_3\gamma_{8}$, \\
\end{tabular}
\end{center}

\begin{center}
\begin{tabular}{ccc}
$s_1\gamma_{9}s_1\gamma_{39}$,  &$s_2\gamma_{9}s_2\gamma_{3}$,  &$s_3\gamma_{9}s_3\gamma_{8}\gamma_{24}$, \\
$s_1\gamma_{10}s_1\gamma_{34}$,  &$s_2\gamma_{10}s_2\gamma_{40}$,  &$s_3\gamma_{10}s_3\gamma_{8}\gamma_{19}$, \\
$s_1\gamma_{11}s_1\gamma_{26}$,  &$s_2\gamma_{11}s_2\gamma_{32}$,  &$s_3\gamma_{11}s_3\gamma_{42}$, \\
$s_1\gamma_{12}s_1\gamma_{18}$,  &$s_2\gamma_{12}s_2\gamma_{24}$,  &$s_3\gamma_{12}s_3\gamma_{42}\gamma_{33}$,\\
\end{tabular}
\end{center}
\begin{center}
\begin{tabular}{ccc}
$s_1\gamma_{14}s_1\gamma_{23}$,  &$s_2\gamma_{14}s_2\gamma_{29}$,  &$s_3\gamma_{14}s_3\gamma_{37}$, \\
$s_1\gamma_{15}s_1\gamma_{33}$,  &$s_2\gamma_{15}s_2\gamma_{39}$,  &$s_3\gamma_{15}s_3\gamma_{5}\gamma_{18}$, \\
$s_1\gamma_{16}s_1\gamma_{28}$,  &$s_2\gamma_{16}s_2\gamma_{34}$,  &$s_3\gamma_{16}s_3\gamma_{5}$, \\
$s_1\gamma_{17}s_1\gamma_{20}$,  &$s_2\gamma_{17}s_2\gamma_{26}$,  &$s_3\gamma_{17}s_3\gamma_{5}\gamma_{34}$, \\
\end{tabular}
\end{center}
\begin{center}
\begin{tabular}{ccc}
$s_1\gamma_{20}s_1\gamma_{17}$,  &$s_2\gamma_{20}s_2\gamma_{23}$,  &$s_3\gamma_{20}s_3\gamma_{2}\gamma_{33}$, \\
$s_1\gamma_{21}s_1\gamma_{27}$,  &$s_2\gamma_{21}s_2\gamma_{33}$,  &$s_3\gamma_{21}s_3\gamma_{2}$, \\
$s_1\gamma_{22}s_1\gamma_{22}$,  &$s_2\gamma_{22}s_2\gamma_{28}$,  &$s_3\gamma_{22}s_3\gamma_{36}$, \\
$s_1\gamma_{27}s_1\gamma_{21}$,  &$s_2\gamma_{27}s_2\gamma_{27}$,  &$s_3\gamma_{27}s_3\gamma_{35}$, \\
$s_1\gamma_{28}s_1\gamma_{16}$,  &$s_2\gamma_{28}s_2\gamma_{22}$,  &$s_3\gamma_{28}s_3\gamma_{32}$.
\end{tabular}
\end{center}

Therefore, $\Gamma$ is a normal subgroup of $\Del$, and consequently, of $\Delp$ as well.
\end{proof}

Later, we will establish that the index $[\Delp:\Gamma]$ equals 504 by evaluating the area of a fundamental domain for $\Gamma$; see \refprop{prop:gammaindex}.

To regard $\Gamma$ as a Fuchsian group, we consider its concrete action on $\disk$.
The group $\Del$ acts on $\disk$, where the generators $s_1, s_2, s_3$ act as reflections along the sides of a hyperbolic triangle.
We fix a triangle with vertices $P_0, P_1, P_2$ using the coordinates:
$P_0 = (0,0)$, $P_1 = r_0(\cos(\frac{5\pi}{14}), \sin(\frac{5\pi}{14}))$, and $P_2 = (0, y_0)$.
Here, $r_0$ is the unique real root of the polynomial $1-6x^2-97x^4-244x^6-97x^8-6x^{10}+x^{12}$ such that $0<r_0<1$.
Likewise, $y_0$ is the unique real root of the polynomial $1-6x^2-48x^4-83x^6-48x^8-6x^{10}+x^{12}$ such that $0<y_0<1$. 
The triangle is illustrated in Figure~\ref{fig:237_tri}.
The generators $s_1$ and $s_2$ act as Euclidean reflections along the line segments connecting the pairs of vertices $P_0$, $P_1$, and $P_0$, $P_2$, respectively.
In a similar manner, $s_3$ acts as a hyperbolic reflection across the geodesic line segment that connects the vertices $P_1$ and $P_2$.
The angles of this hyperbolic triangle are $\pi/7$ at $P_0$, $\pi/2$ at $P_1$, and $\pi/3$ at $P_2$.

These values can be found from a relationship between the lengths of the sides and the angles of a hyperbolic triangle.
For a hyperbolic triangle with vertices $A$, $B$, and $C$, the hyperbolic law of cosine is given by the equation
\begin{equation}
\cosh(a) = \cosh(b) \cosh(c) - \sinh(b) \sinh(c) \cos(\alpha).
\label{eq:hyperbolic_cosine}
\end{equation}
Here, \(a\), \(b\), and \(c\) represent the lengths of the sides of the triangle opposite to vertices $A$, $B$, and $C$ respectively, 
whiel \(\alpha\) denotes the angle at vertex $A$.

\begin{figure}[h]
    \centering
    \includegraphics[width=0.5\textwidth]{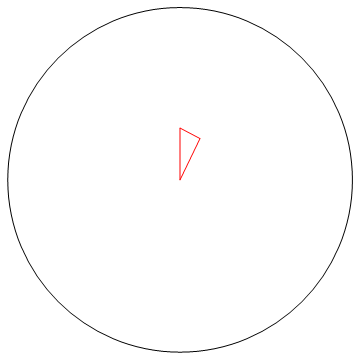}
    \caption{Fundamental hyperbolic triangle for $\Del$.}
    \label{fig:237_tri}
\end{figure}

Henceforth, with the given action, we consider both $\Delp$ and $\Gamma$ as subgroups of $\PSU(1,1)$, the group of orientation-preserving isometries of $\disk$.
As elements of $\PSU(1,1)$, we have the following generators $a$ and $b$ of $\Delp$:
\begin{align}
a &  = s_3s_2 = 
\pm 
\pmat{\alpha}{\beta}{\conj{\beta}}{\conj{\alpha}}
, \label{eq:matrixA} \\
b &  = s_2s_1 =
\pm 
\pmat
{\cos \left(\frac{\pi }{7}\right)+\i \sin \left(\frac{\pi }{7}\right)} {0}
{0} {\cos \left(\frac{\pi }{7}\right)-\i \sin \left(\frac{\pi }{7}\right)}
, \label{eq:matrixB}
\end{align}
where $\alpha$ is a root of 
$1 - 7 x + 21 x^2 - 35 x^3 + 35 x^4 - 21 x^5 + 7 x^6$ such that
$\alpha \approx 0.5+1.03826\i$, and 
$\beta$ is a root of $-1 + 7 x^4 + 7 x^6$ such that $\beta \approx 0.5727$.
Note that we have $a^3 = b^7 = (ab)^2 = 1$.

\subsection{Fundamental domain for $\Gamma$}
It is worth noting that the motivation behind adopting the above action of $\Del$ is that it provides the classical tessellation of a fundamental domain, given as a regular 14-gon inside $\disk$, for the Klein quartic. 
This is illustrated in Figure~\ref{fig:237_klein}.
\begin{figure}[h]
    \centering
    \includegraphics[width=0.5\textwidth]{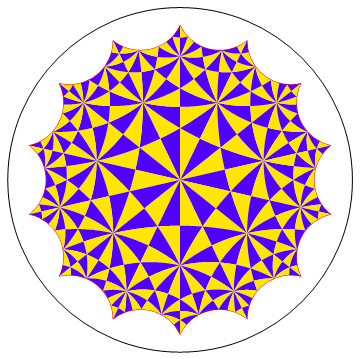}
    \caption{Tessellation of a fundamental domain for the Klein quartic}
    \label{fig:237_klein}
\end{figure}
The main goal of this subsection is to determine an explicit fundamental domain for $\Gamma$.
To achieve this, we will employ Poincar\'{e}'s polyhedron theorem as presented by Maskit~\cite{maskit1971}.
Recall that a closed set $D$ is called a \textit{fundamental domain} for a Fuchsian group $G$ if:
\begin{enumerate}
    \item No two distinct points of the interior of $D$ are equivalent under $G$, and
    \item Every point in $\disk$ is equivalent under $G$ to some point of $D$.
\end{enumerate}

For an element \(g = \pm \pmat{a}{b}{c}{d} \in \PSU(1,1)\) with \(c \neq 0\), 
the \textit{isometric circle} \(I(g)\) of \(g\) is defined by:
\begin{equation}
I(g) = \{z \in \mathbb{C} : |cz + d| = 1\}.
\label{eq:isometric_circle}
\end{equation}
As a Euclidean circle, its center is given by \(-d / c\) and it has a radius of \(1 / |c|\).
On $I(g)$, $g$ acts as a Euclidean isometry.
For $z\in \disk$, we have
\begin{equation}
z\in I(g) \iff \distD(z, 0) = \distD(z, g^{-1} 0).
\label{eq:isometric_dist}
\end{equation}
Furthermore, this defines a geodesic in $\disk$.
This can be seen by considering the triangle with vertices at the center of $\disk$, the center of $I(g)$, and the intersection point of $I(g)$ with the unit circle.
It turns out to be a right triangle.
For further details on isometric circles, refer to~\cite[Sections 4.1 and 7.36]{MR0698777}.

The \textit{exterior} of \(I(g)\), denoted by $\ext(I(g))$, is the subset of \(\disk\) given by:
\[
\ext(I(g)) = \{z \in \disk : |cz + d| > 1\}.
\]

\begin{lemma}
Two isometric circles $I(\gamma_i)$ and $I(\gamma_j)$ with $i\neq j$ have a non-empty intersection within the unit disk $\disk$ 
if and only if $j = i\pm 1 \mod 42$,
and each intersection consists of a single point.
\end{lemma}

\begin{proof}
By utilizing \eqref{eq:matrixA} and \eqref{eq:matrixB}, we derive an exact matrix expression for each $\gamma_i$.
Applying \eqref{eq:isometric_circle}, we can algebraically determine the intersection of two distinct isometric circles, $I(\gamma_i)$ and $I(\gamma_j)$, within the unit disk $\disk$.
Hence, this computation is straightforward and exact.

Whenever there is a non-empty intersection of $I(\gamma_i)$ and $I(\gamma_j)$ within the unit disk, it must consist of just one point.
If two such isometric circles intersect at multiple points, then they must be identical as they are also geodesics in $\disk$.
\end{proof}

A visual representation of the isometric circles, $I(\gamma_i)$, can be found in Figure~\ref{fig:iso_circles}.
\begin{figure}[h]
    \centering
    \includegraphics[width=0.9\textwidth]{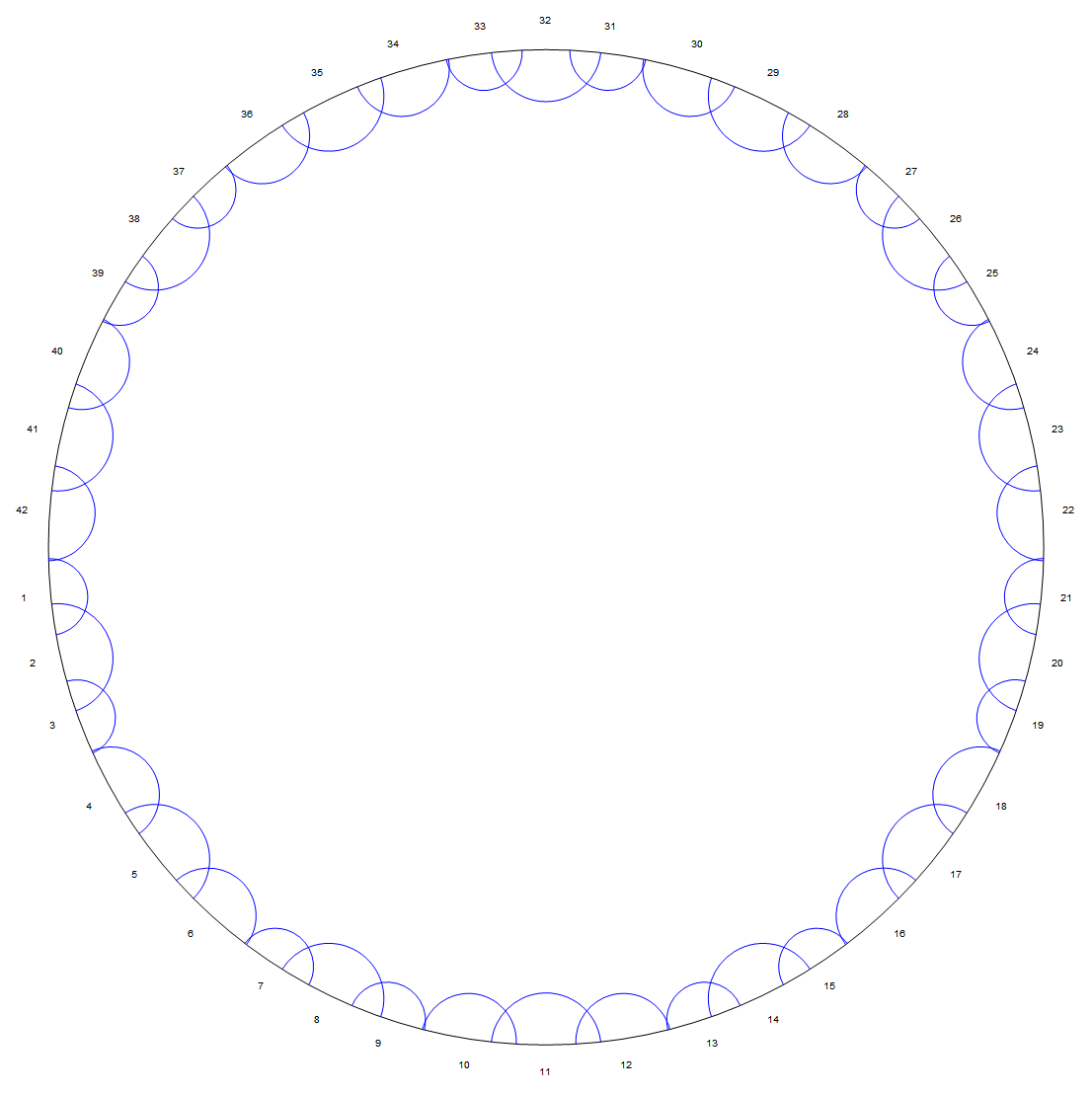}
    \caption{Isometric circles of $\gamma_i$}
    \label{fig:iso_circles}
\end{figure}
For each $i=1,\dots, 42$, let $v_i\in \disk$ be the unique point of the intersection $I(\gamma_i)\cap I(\gamma_{i-1})$, where we set $I(\gamma_{0}) = I(\gamma_{42})$ as before.

We define a compact subset $\F$ of the unit disk $\disk$ as the region bounded by the isometric circles $I(\gamma_i)$ for $1 \leq i \leq 42$:
\begin{equation}
\F = \text{closure of }\bigcap_{1\leq i\leq 42} \ext(I(\gamma_i)).
\label{eq:Fdef}
\end{equation}
Then $\F$ is a hyperbolic polygon with 42 sides with vertices $v_i, \, 1 \leq i \leq 42$.
We denote the geodesic segment connecting $v_i$ and $v_{i+1}$ by $s_i$, hence a subset of $I(\gamma_i)$, and call it a \textit{side} of $\F$.
For a visual representation of $\F$, \( s_i \) and \( v_i \), refer to Figure~\ref{fig:fricke_domain}.

\begin{figure}[h]
    \centering
    \includegraphics[width=0.9\textwidth]{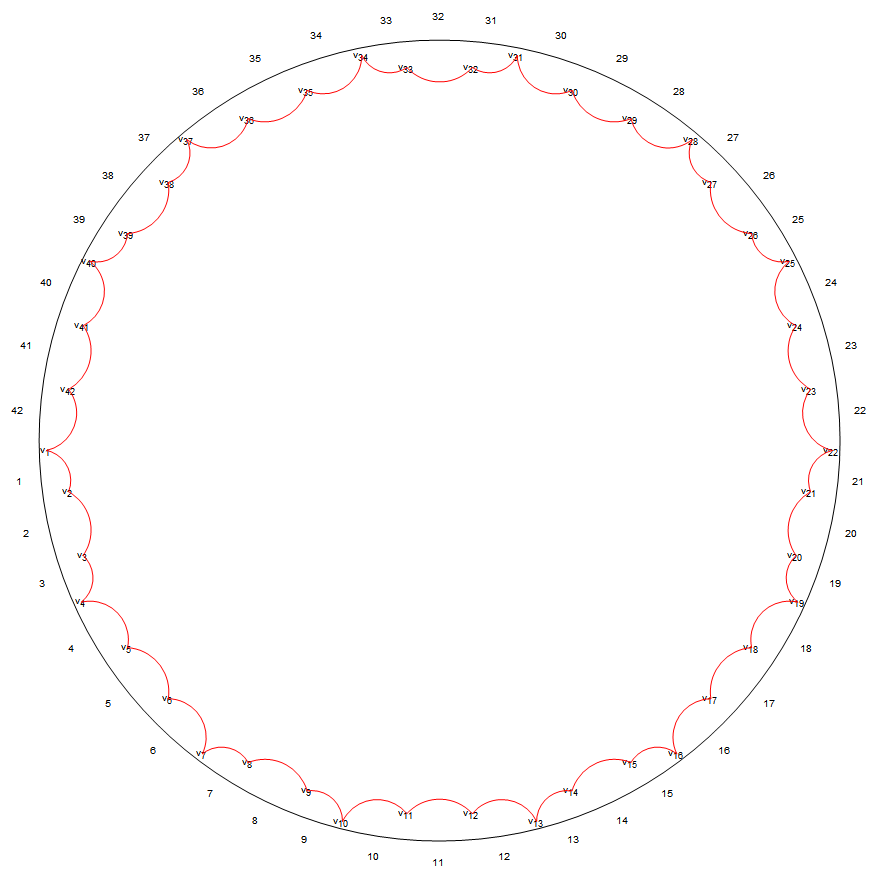}
    \caption{Hyperbolic polygon $\F$ and its vertices $\{v_i\}$}
    \label{fig:fricke_domain}
\end{figure}

We now consider a symmetry property of $\F$.
It is evident that $b$, as defined in \eqref{eq:matrixB}, acts as a rotation around the origin by an angle of $2\pi/7$.
We can express \eqref{eq:gamma_recursive} as $\gamma_{i} = b\gamma_{i-6}b^{-1}, \, i\in \Z$.
Using the invariance property of isometric circles,
\[
I(hgh^{-1}) = h(I(g)),\, g,h \in \PSU(1,1),
\]
we can establish the following:
\begin{prop} \label{prop:Fsym}
The polygon $\F$ is invariant under the action of the cyclic group of rotations generated by $b$.
\end{prop}

From \refprop{prop:Fsym}, we have
\[
v_{j} = e^{2\pi \i /7}v_{j-6}.
\]
Table~\ref{tab:vcoor} provides the explicit coordinates of the points $v_j = x_j+  y_j\i$ for $1\leq j \leq 6$.

\begin{table}[ht]
\centering
\begin{tabular}{|c|c|p{10cm}|}
\hline
Coordinate & Numerical value & Minimal polynomial \\
\hline
$x_1$ & $-0.9809$ & $-7 - 336 x^2 - 3584 x^4 + 4096 x^6$ \\
\hline
$y_1$ & $-0.0246$ & $-1 + 1648 x^2 + 5632 x^4 + 4096 x^6$ \\
\hline
$x_2$ & $-0.9263$ & $1 + 335 x^2 + 19055 x^4 + 96809 x^6 - 8896 x^8 - 376832 x^{10} + 262144 x^{12}$ \\
\hline
$y_2$ & $-0.1277$ & $49 - 5047 x^2 + 113631 x^4 + 688415 x^6 + 1331008 x^8 + 1032192 x^{10} + 262144 x^{12}$ \\
\hline
$x_3$ & $-0.8900$ & $28561 + 148889 x^2 + 119344 x^4 - 382383 x^6 - 445248 x^8 + 253952 x^{10} + 262144 x^{12}$ \\
\hline
$y_3$ & $-0.2869$ & $49 - 441 x^2 - 5096 x^4 + 20447 x^6 + 192192 x^8 + 401408 x^{10} + 262144 x^{12}$ \\
\hline
$x_4$ & $-0.8944$ & $-1183 - 2576 x^2 + 1792 x^4 + 4096 x^6$ \\
\hline
$y_4$ & $-0.4034$ & $-1 - 144 x^2 + 256 x^4 + 4096 x^6$ \\
\hline
$x_5$ & $-0.7792$ & $28561 + 116948 x^2 - 80583 x^4 - 471535 x^6 - 110144 x^8 + 454656 x^{10} + 262144 x^{12}$ \\
\hline
$y_5$ & $-0.5170$ & $49 + 980 x^2 + 3185 x^4 - 26593 x^6 - 78400 x^8 + 200704 x^{10} + 262144 x^{12}$ \\
\hline
$x_6$ & $-0.6774$ & $28561 + 51883 x^2 - 195334 x^4 - 389551 x^6 + 226304 x^8 + 684032 x^{10} + 262144 x^{12}$ \\
\hline
$y_6$ & $-0.6446$ & $49 + 1421 x^2 + 7938 x^4 - 4641 x^6 - 89600 x^8 - 28672 x^{10} + 262144 x^{12}$ \\
\hline
\end{tabular}
\caption{Coordinates for $v_j = x_j+  y_j \i,\, 1 \leq j \leq 6$}
\label{tab:vcoor}
\end{table}

For every side $s$ of $\F$, there exists a corresponding side $s'$ and an isometry $A(s,s')\in \Gamma$
such that $A(s,s')$ maps side $s$ onto side $s'$, $(s')' = s$ and $A(s', s) = (A(s, s'))^{-1}$.
These pairs of sides $(s,s')$ can be easily determined using \eqref{eq:sidepairings}.
For example, $A(s_1,s_{33}) = \gamma_1$ and $A(s_{33},s_1)=\gamma_{33}$.
Such a collection $\Phi=\{A(s,s')\}$ of isometries is called a \textit{side-paring} of $\F$.

Now we introduce the notion of \textit{cycle} of vertices following~\cite{maskit1971}.
Let $z_1$ be some vertex, say, $v_i$ of $\F$.
There are two sides of $\F$, namely, $s_{i-1}$ and $s_i$ which meet at $z_1$.
We choose one of these, denoting it as $t_1$.
There is a corresponding side $t_1'$ and an isometry $A_1 = A(t_1,t_1')$.
The next vertex, $z_2$, is defined as the image of $z_1$ under the action of $A_1$, i.e., $z_2 = A_1(z_1)$.
There is a unique other side $t_2$, i.e., $t_2\neq t_1'$ which has $z_2$ as one of its endpoint. 
There is a corresponding side $t_2'$ and a generator $A_2 = A(t_2, t_2')$.
We can continue this process iteratively: $z_3 = A_2(z_2)$ and let $s_3$ be the unique other side which has $z_3$ as an endpoint, and so on.
This procedure generates sequences: $\{z_n\}$ for vertices, $\{(t_n, t_n')\}$ for pairs of sides, and $\{A_n\}$ for $\Phi$.
We define the \textit{period} as the smallest positive integer $m$ for which all three sequences are periodic with period $m$.
Then the ordered list of vertices $(z_1, \ldots, z_m)$ is called a cycle of vertices.

We can find cycles starting from a given vertex by following the definition.
For example, consider a cycle that starts with the vertex $z_1: =  v_{2}$.
This vertex is the intersection point of sides $s_1$ and $s_2$.
By setting $t_1: = s_1$, we find $t_1' = s_{33}$ and $A_1 = A(s_1, s_{33})=\gamma_1$.
With these choices, we can determine the subsequent terms in the following way:
\begin{itemize}
    \item \( z_2 = A_1(z_1) = v_{33} \), and \( A_2 = A(s_{32}, s_{11}) = \gamma_{32} \).
    \item \( z_3 = A_2(z_2) = v_{11} \), and \( A_3 = A(s_{10}, s_{24}) = \gamma_{10} \).
    \item \( z_4 = A_3(z_3) = v_{24} \), and \( A_4 = A(s_{23}, s_{2}) = \gamma_{23} \).
    \item \( z_5 = A_4(z_4) = v_{2} \), and \( A_5 = A(s_{1}, s_{33}) = \gamma_{1} \)
\end{itemize}
After this point, the cycle becomes periodic.
In this specific example, we have a cycle with a period of 4.

There exist cycles of both period 4 and period 14 within the vertices of $\F$.
These vertices can be grouped into 8 distinct sets, with each set containing members that belong to a specific cycle.
Here is the list of all cycles, where each pair of adjacent vertices in a cycle is connected an arrow, and each arrow is labeled with the corresponding side-pairing transformation.
We will begin with the cycles having a period of 4, listed as follows:
\begin{alignat}{6}
v_{2}
& \xrightarrow{\gamma_{1}} v_{33}
& \xrightarrow{\gamma_{32}} v_{11}
& \xrightarrow{\gamma_{10}} v_{24}
& \xrightarrow{\gamma_{23}} v_{2},\qquad  &&
v_{2}
& \xleftarrow{\gamma_{33}} v_{33}
& \xleftarrow{\gamma_{11}} v_{11}
& \xleftarrow{\gamma_{24}} v_{24}
& \xleftarrow{\gamma_{2}} v_{2}, \notag \\
v_{3}
& \xrightarrow{\gamma_{3}} v_{14}
& \xrightarrow{\gamma_{14}} v_{36}
& \xrightarrow{\gamma_{36}} v_{23}
& \xrightarrow{\gamma_{23}} v_{3},\qquad &&
v_{3}
& \xleftarrow{\gamma_{13}} v_{14}
& \xleftarrow{\gamma_{35}} v_{36}
& \xleftarrow{\gamma_{22}} v_{23}
& \xleftarrow{\gamma_{2}} v_{3},\notag \\
v_{5}
& \xrightarrow{\gamma_{5}} v_{27}
& \xrightarrow{\gamma_{27}} v_{38}
& \xrightarrow{\gamma_{38}} v_{18}
& \xrightarrow{\gamma_{18}} v_{5},\qquad &&
v_{5}
& \xleftarrow{\gamma_{26}} v_{27}
& \xleftarrow{\gamma_{37}} v_{38}
& \xleftarrow{\gamma_{17}} v_{18}
& \xleftarrow{\gamma_{4}} v_{5}, \notag \\
v_{6}
& \xrightarrow{\gamma_{6}} v_{35}
& \xrightarrow{\gamma_{35}} v_{15}
& \xrightarrow{\gamma_{15}} v_{26}
& \xrightarrow{\gamma_{26}} v_{6},\qquad &&
v_{6}
& \xleftarrow{\gamma_{34}} v_{35}
& \xleftarrow{\gamma_{14}} v_{15}
& \xleftarrow{\gamma_{25}} v_{26}
& \xleftarrow{\gamma_{5}} v_{6}, \notag \\
v_{8}
& \xrightarrow{\gamma_{8}} v_{30}
& \xrightarrow{\gamma_{30}} v_{17}
& \xrightarrow{\gamma_{17}} v_{39}
& \xrightarrow{\gamma_{39}} v_{8},\qquad &&
v_{8}
& \xleftarrow{\gamma_{29}} v_{30}
& \xleftarrow{\gamma_{16}} v_{17}
& \xleftarrow{\gamma_{38}} v_{39}
& \xleftarrow{\gamma_{7}} v_{8}, \notag \\
v_{9}
& \xrightarrow{\gamma_{8}} v_{29}
& \xrightarrow{\gamma_{28}} v_{42}
& \xrightarrow{\gamma_{41}} v_{20}
& \xrightarrow{\gamma_{19}} v_{9},\qquad &&
v_{9}
& \xleftarrow{\gamma_{29}} v_{29}
& \xleftarrow{\gamma_{42}} v_{42}
& \xleftarrow{\gamma_{20}} v_{20}
& \xleftarrow{\gamma_{9}} v_{9}, \notag \\
v_{12}
& \xrightarrow{\gamma_{11}} v_{32}
& \xrightarrow{\gamma_{31}} v_{21}
& \xrightarrow{\gamma_{20}} v_{41}
& \xrightarrow{\gamma_{40}} v_{12},\qquad &&
v_{12}
& \xleftarrow{\gamma_{32}} v_{32}
& \xleftarrow{\gamma_{21}} v_{21}
& \xleftarrow{\gamma_{41}} v_{41}
& \xleftarrow{\gamma_{12}} v_{12}.
\label{eq:vertex_cycle1}
\end{alignat}

Here is the list of cycles with a period of 14:
\begin{alignat}{2}
v_{1} & \xrightarrow{\gamma_{42}} v_{28} \xrightarrow{\gamma_{27}} v_{37} \xrightarrow{\gamma_{36}} v_{22} \xrightarrow{\gamma_{21}} v_{31} \xrightarrow{\gamma_{30}} v_{16} \xrightarrow{\gamma_{15}} v_{25} \xrightarrow{\gamma_{24}} v_{10} \notag \\
&
\xrightarrow{\gamma_{9}} v_{19}
\xrightarrow{\gamma_{18}} v_{4}
\xrightarrow{\gamma_{3}} v_{13}
\xrightarrow{\gamma_{12}} v_{40} \xrightarrow{\gamma_{39}} v_{7}
\xrightarrow{\gamma_{6}} v_{34}
\xrightarrow{\gamma_{33}} v_{1}, \notag \\
v_{1} &
\xleftarrow{\gamma_{28}} v_{28} 
\xleftarrow{\gamma_{37}} v_{37}
\xleftarrow{\gamma_{22}} v_{22} 
\xleftarrow{\gamma_{31}} v_{31}
\xleftarrow{\gamma_{16}} v_{16}
\xleftarrow{\gamma_{25}} v_{25}
\xleftarrow{\gamma_{10}} v_{10} \notag\\
&
\xleftarrow{\gamma_{19}} v_{19}
\xleftarrow{\gamma_{4}} v_{4} 
\xleftarrow{\gamma_{13}} v_{13}
\xleftarrow{\gamma_{40}} v_{40}
\xleftarrow{\gamma_{7}} v_{7}
\xleftarrow{\gamma_{34}} v_{34}
\xleftarrow{\gamma_{1}} v_{1}.
\label{eq:vertex_cycle2}
\end{alignat}

For each vertex $v_i$, suppose that two sides $s_{i-1}$ and $s_{i}$ make an angle $\alpha(v_i)$ at $v_i$ measured from inside the polygon.
Using the algebraic coordinates for $v_i$ provided in Table~\ref{tab:vcoor} and 
applying formula \eqref{eq:hyperbolic_cosine}, we can compute $\cos(\alpha(v_i))$.
It turns out that $\cos(\alpha(v_i))$ is an algebraic number.
Due to the symmetry of $\F$, as explained in \refprop{prop:Fsym}, we only need to calculate $\alpha(v_i)$ for $1\leq i \leq 6$. The precise values for these angles are presented in Table~\ref{tab:vertex_angle}.

\begin{table}[h]
\centering
\begin{tabular}{|c|c|c|}
\hline
$i$ & Minimal polynomial of $\cos(\alpha(v_i)))$ & Numerical value of $\alpha(v_i)/\pi$ \\
\hline
1 & $1-4x-4x^2+8x^3$ & 0.142857 \\
\hline
2 & $-1+56x^2-336x^4+448x^6$ & 0.545452 \\
\hline
3 & $-1+56x^2-336x^4+448x^6$ & 0.545452 \\
\hline
4 & $1-4x-4x^2+8x^3$ & 0.142857 \\
\hline
5 & $-1+56x^2-336x^4+448x^6$ & 0.454548 \\
\hline
6 & $-1+56x^2-336x^4+448x^6$ & 0.454548 \\
\hline
\end{tabular}
\caption{Angles at vertices of $\F$}
\label{tab:vertex_angle}
\end{table}

\begin{lemma}\label{lemma:cycle_angle_sum}
For every cycle \( (z_1,\dots, z_m) \) of vertices of \( \F \), the sum of angles is \( 2\pi \):
\[
\sum_{i=1}^{m}\alpha(z_i) = 2\pi.
\]
\end{lemma}

\begin{proof}
Note that if $i\equiv j \mod 6$, then $\alpha(v_i) = \alpha(v_{j})$.
Let us establish the proof by considering different cases.

For a cycle $(z_1,z_2,z_3,z_4)$ of period 4 as given in \eqref{eq:vertex_cycle1}, 
we have:
\[
\sum_{i=1}^{4}\alpha(z_i) = 
\alpha(v_2)+\alpha(v_3)+\alpha(v_5)+\alpha(v_6)
=2(\alpha(v_2)+\alpha(v_5))
=2\pi.
\]
And for a cycle $(z_1,\dots,z_{14})$ of period 14 in \eqref{eq:vertex_cycle2}, 
we have:
\[
\sum_{i=1}^{14}\alpha(z_i) = 
7(\alpha(v_1)+\alpha(v_4))
=2\pi.
\]    
\end{proof}

\begin{thm} \label{thm:fund}
The polygon $\F$ is a fundamental domain for $\Gamma$.
\end{thm}
\begin{proof}
Let us employ Poincar\'{e}'s theorem as presented in~\cite{maskit1971}:
if for every cycle $(z_1,\dots, z_m)$, there is an integer $\nu$ for which
\[
\sum_{i=1}^{m}\alpha(z_i) = \frac{2\pi}{\nu},
\]
then $\F$ is a fundamental domain for the group generated by side-pairing $\Phi$, which is nothing but $\Gamma$.

This reduces our task to calculating the angle sum for every cycle.
The theorem is thus deduced from Lemma ~\ref{lemma:cycle_angle_sum}.
\end{proof}

\begin{prop}\label{prop:gammaindex}
The index $[\Delp:\Gamma]$ is 504.
\end{prop}
\begin{proof}
From Table~\ref{tab:vertex_angle}, we have
\[
\sum_{i=1}^{42}\alpha(z_i) = 7(2+2/7)\pi = 16\pi.
\]
This implies that the area of $\F$ is $24\pi$.
Therefore, $[\Delp:\Gamma]=504$, given that the fundamental domain of $\Delp$ has area $\pi/21$, which is twice the area of the fundamental triangle for $\Del$  illustrated in Figure~\ref{fig:237_tri}.
\end{proof}

\begin{cor}
The group $\Gamma$ is the normal subgroup of $\Delp$ generated by $\gamma_{15}$.
\end{cor}
\begin{proof}
Let $\Gamma'$ be the normal subgroup of $\Delp$ with index 504, which was considered in Subsection~\ref{subsec:hurwitzsurfaces}.
It is the normal subgroup of $\Delp$ generated by:
\begin{align}
\gamma & = (a^{-1}b^3a^{-1}ba^{-1}b^3)^2 \notag \\
& = (s_2s_3s_2s_1s_2s_1s_2s_1s_2s_3s_2s_1s_2s_3s_2s_1s_2s_1s_2s_1)^2 \notag \\
&= s_2s_3s_2s_1s_2s_1s_2s_1s_3s_2s_1s_2s_1s_3s_2s_1s_2s_1s_2s_3s_2s_1s_2s_1s_2s_1s_3s_2s_1s_2s_1s_3s_2s_1s_2s_1 \notag \\
&= \gamma_{15}. \label{eq:gamma15}
\end{align}
Thus, we have $\Gamma' \subseteq \Gamma$.
Considering both subgroups have the same index, we deduce \( \Gamma' = \Gamma \).
\end{proof}

Recall that for $w\in \disk$ with trivial stabilizer in $\Gamma$, the \textit{Dirichlet domain} centered at $w\in \disk$ is defined by
\[
D(w)=\{z \in \disk: \distD(z, w) \leq \distD(z, g w) \text { for all } g \in \Gamma\} .
\]
It is well-known that $D(w)$ is a fundamental domain for $\Gamma$; see, for example,~\cite[Section 9.4]{MR0698777}.
From \eqref{eq:isometric_dist}, the domain $D(0)$ can be described as the closure in $\disk$ of 
\[
\bigcap_{g \in \Gamma \backslash\{1\}} \ext(I(g)).
\]

\begin{thm}\label{thm:dirichlet}
The polygon $\F$ is the Dirichlet domain centered at $0$.    
\end{thm}
\begin{proof}
Referring to \eqref{eq:Fdef}, one can clearly see that \(D(0)\) is a subset of \(\F\). 
It is sufficient to verify that every interior point of \(\F\) lies within \(D(0)\). 
Then, considering that both sets are closed, it implies that the two sets are the same.

Let $z\in \F\backslash D(0)$ be an interior point of $\F$.
Given that $D(0)$ is a fundamental domain for $\Gamma$, there exists a transformation $\gamma\in \Gamma$ such that $\gamma z\in D(0)\subseteq \F$.
Consequently, both \(z\) and \(\gamma z\) reside in \(\F\), which implies that they are boundary points of $\F$.
From this, we conclude that \(D(0)\) and \(\F\) are indeed identical.
\end{proof}

\subsection{Length spectrum of the Fricke-Macbeath surface}
In Vogeler's study~\cite{Vogeler2003}, an in-depth analysis of closed geodesics on Hurwitz surfaces is presented, with a special focus on the length spectrum of Hurwitz surfaces of low genera.
Of particular relevance to us, Vogeler offers an exhaustive table~\cite[Table C.2]{Vogeler2003} enumerating the short primitive closed geodesics on the Fricke-Macbeath surface, along with their corresponding lengths and multiplicities.
This table comprehensively covers geodesics with lengths below 14.49~\cite[p. 57]{Vogeler2003}.
For those interested on understanding the method used, refer to~\cite[Section 2.3]{Vogeler2003}.

In order to compute the multiplicity of \(\la_1\) for the Fricke-Macbeath surface using the Selberg trace formula, we need the conjugacy classes of the two shortest lengths as input data.
The shortest classes have an approximate length of \(5.796298891\), while
the next shortest classes are about \(8.303943294\) in length.

To employ interval arithmetic in the trace formula, we need a precise error bound for these lengths.
To convert an entry from~\cite[Table C.2]{Vogeler2003} into a conjugacy class of $\Gamma$, one can follow these steps:
\begin{itemize}
    \item A proper label for a geodesic is a sequence consisting of the letters $R$ and $L$. The sequence begins with either one or two $R$’s, followed by one or two $L$’s, and continues in this alternating pattern. Further details can be found in~\cite[p. 19]{Vogeler2003}.
    \item This proper label can be converted into a code. A code in the table is a sequence of 1’s and 2’s. Each number in this sequence denotes the count of consecutive $R$’s or $L$’s. As an illustration, the code $(1 1 2 2)$ translates to the proper label $RLRRLL$. This encoding is described in~\cite[p. 23]{Vogeler2003}.
    \item Now, one should substitute the letter $R$ with the expression $a^{-1}b$ and the letter $L$ with $b$. This transformation yields a word consisting of $a$ and $b$. The process is explained in~\cite[p. 32]{Vogeler2003}.    
    \item Lastly, with the aid of \eqref{eq:matrixA} and \eqref{eq:matrixB}, an exact matrix representation can be obtained for this word in $a$ and $b$.    
\end{itemize}

Consider the first entry with code $(1222)$ and $n=2$ in~\cite[Table C.2]{Vogeler2003}.
The code \((1222)\) which translates to the proper label $(RLLRRLL)$, and since \(n = 2\), this label repeats twice:
\[
(RLLRRLL)(RLLRRLL).
\]
The geodesic with this label corresponds to the conjugacy class of $\Gamma$ containing
\begin{align*}
\gamma
&= (a^{-1}b^3a^{-1}ba^{-1}b^3)(a^{-1}b^3a^{-1}ba^{-1}b^3) \\
&= \gamma_{15} \text{, as checked in } \eqref{eq:gamma15}.
\end{align*}
Let $z_0$ be half the trace of $\gamma$, whose minimal polynomial is $1-x-9 x^2+x^3$.
We can easily check that 
\[z_0 \in (9.09783467903, 9.09783467905).\]
The length of $\gamma$ is $\ell(\gamma) = 2 \arccosh(z_0) = 5.796298891\dots$.
This length is significant as it represents the shortest geodesic length on the Fricke-Macbeath surface $\X$.
This shortest geodesic length, also termed the \textit{systole} and denoted by \( \sys(\X) \), is defined as:
\[
\sys(\X) = \min \{ \ell(\gamma) : \{\gamma\} \in \calC(\Gamma) \}.
\]

Similarly, the geodesic labeled by the code $(11111112)$ with $n=2$ corresponds to a conjugacy class with length 
$2 \arccosh(z_1) = 8.303943294\dots$ with $z_1$ having minimal polynomial $-1-25 x-31 x^2+x^3$.
We can easily check that 
\[z_1 \in (31.78746324455, 31.78746324457).\]

Finally, note that the multiplicity of a geodesic listed in Vogeler's table should be doubled when we are counting the number of corresponding conjugacy classes.
This adjustment is necessary because the table enumerates unoriented geodesics; see \cite[Definition 9]{Vogeler2003}.

The preceding discussion can be summarized in the following proposition:
\begin{prop}\label{prop:shortgeo}
Let $z_0$ and $z_1$ be the above algebraic numbers.
The conjugacy classes with the minimum length in $\calC(\Gamma)$ has length $2\arccosh(z_0)$, and there are 252 classes. 
There are 504 conjugacy classes that are the second shortest, each with a length $2\arccosh(z_1)$.
\end{prop}
The precise intervals within which $z_0$ and $z_1$ lie, as well as the above multiplicities play a crucial role in the computations utilizing the Selberg trace formula with interval arithmetic.

\section{Multiplicity of the first eigenvalue}
\subsection{Overview of the proof}
In this section, we will investigate the multiplicity $m_1$ of the first eigenvalue $\la_1$ of the Laplacian on the Fricke-Macbeath surface $\X=\Gamma\backslash\disk$.
Specifically, we will prove the following:
\begin{theorem}\label{thm:mult7}
The multiplicity of the first eigenvalue $\la_1$ of the Fricke-Macbeath surface is 7, and $\la_1\in [1.23, 1.26]$.
\end{theorem}

To achieve this, we begin with the following key lemma:
\begin{lemma}\label{prop:mult7}
The multiplicity of the first eigenvalue $\la_1$ is at least 7.
\end{lemma}
\begin{proof}
By~\cite[Lemma 9.2]{FBP23}, no 1-dimensional subrepresentation of $\Aut(\X)$ can occur in the eigenspace $V_{\la_1}$ associated to $\la_1$.
From the character table of $\Aut(\X)$ given in Table~\ref{table:character_table}, we know that $\Aut(X)$ has irreducible representations of dimensions 1,7,8, and 9.
As a result, $V_{\la_1}$ contains a direct summand with a dimension of at least 7.
\end{proof}

Our main tool for the proof of Theorem \ref{thm:mult7} is the Selberg trace formula \eqref{eq:STF1}.
Following the notation in~\cite{FBP21} for the Selberg trace formula, we define
\begin{align*}
\calS &= \sum_{n=0}^{\infty} h(r_{n}),\\
\calI &= 2(g(X)-1)\int_0^\infty r h(r) \tanh (\pi r)\, dr,\\
\calG &= \frac{1}{\sqrt{2\pi}}\sum_{\{\gamma\} \in \calC(\Gamma)} \frac{\La(\gamma)}{2 \sinh(\ell(\gamma)/2)} g(\ell(\gamma)),\\
\end{align*}
so that the Selberg trace formula is compactly expressed as \(\mathcal{S} = \mathcal{I} + \mathcal{G}\).
Here, \(g(X)\) denotes the genus of a surface \(X\), which is 7 in our case.
We have also utilized the relation \(\mathrm{area}(X) = 4\pi (g(X) - 1)\) and the even nature of the function \(h\) to refine the expression for \(\mathcal{I}\).

For our proof of Theorem~\ref{thm:mult7}, we will divide the interval $[0,1.26]$ into four subintervals and establish the following:
\begin{itemize}
    \item $\la_1 \notin [0, 0.25]$ (\refprop{prop:subinterval1})
    \item $\la_1 \notin [0.25, 1]$ (\refprop{prop:subinterval2})
    \item $\la_1 \notin [1, 1.23]$ (\refprop{prop:subinterval3})
    \item $\la_1 \in [1.23, 1.26]$ (\refprop{prop:subinterval4}) and its multiplicity is 7 (\refprop{prop:subinterval4-2})
\end{itemize}
We will choose suitable test functions $h,g$ to derive the desired conclusions.
These test functions have a particular form with many adjustable parameters.
Hence the challenge in applying the Selberg trace formula lies in finding the optimal parameters for these test functions.

In subsections~\ref{subsec:subinterval1} and \ref{subsec:subinterval2}, 
$h(\xi)$ will be given as a linear combination of functions of the form
$\left(\frac{\sin(\delta \xi / 2)}{\delta \xi / 2}\right)^{4}\cos(n \delta \xi)$ for some fixed $\delta\in \R$ with $n\in \Z$.
Such $h$ is an entire function and satisfies the decay properties as given in \eqref{eq:hdecay}.
Its Fourier transform $g$ is an even function with compact support.
Such $h,g$ form an admissible transform pair.
These type of functions were used in~\cite[Section 4]{MR2338122} and~\cite[Section 5]{MR4186122} to verify 
the Selberg eigenvalue conjecture for many congruence subgroups of small conductor.

In subsections~\ref{subsec:subinterval3} and \ref{subsec:subinterval4}, we will use test functions involving the \emph{Hermite polynomials} $\{H_n\}_{n\in \Z_{\geq 0}}$.
They are orthogonal polynomials with weight $e^{-x^2}$ on the real line $\R = (-\infty, \infty)$.
They can be also defined by the Rodrigues formula
\[
H_n(x) = (-1)^n e^{x^2} \frac{d^n}{dx^n}(e^{-x^2}).
\]
For even and odd values of \(n\), \(H_n\) is respectively even and odd, satisfying the relation:
\[
H_n(-x) = (-1)^n H_n(x).
\]
They have a close relationship with certain Laguerre polynomials $\{L_n^{(\alpha)}(x)\}$.
More specifically, the Hermite polynomials of even order satisfies
\[
H_{2n}(x) = (-1)^n 4^{n} n! L_n^{(-1/2)}(x^2).
\]
Let $h_n(x) = H_n(x) e^{-x^2/2}$.
They are eigenfunctions of the Fourier transform with eigenvalues $\pm 1, \pm \i$.
In particular, we have $\what{h_{2n}}(\xi) = (-1)^n h_{2n}(\xi)$.

Our test functions $h,g$ will of the form:
\begin{align}
g(x) &= \left(\sum_{n=0}^N s_n L_n^{(-1/2)}(x^2) \right) e^{-x^2/2} =: u(x^2) e^{-x^2/2}, \notag \\
h(\xi) &=  \left(\sum_{n=0}^N (-1)^n s_n L_n^{(-1/2)}(\xi^2) \right) e^{-\xi^2/2} =: v(\xi^2) e^{-\xi^2/2}. \label{eq:gh_hermite}
\end{align}
The rational numbers $\{s_n\}$ in \eqref{eq:gh_hermite} are what we specify in concrete applications.
In this instance, \( h \) is an even entire function that exhibits rapid decay, readily meeting the criteria set by \eqref{eq:hdecay}. 
Thus, \( h \) and \( g \) together make up an admissible transform pair.
Note that $u$ and $v$ are polynomials with rational coefficients.

These test functions were used in the sphere packing problem~\cite[Section 7]{MR1973059}.
They were first used to study the spectrum of hyperbolic surfaces in~\cite{FBP21}.
They prove particularly valuable when constructing test functions that satisfy some sign conditions for both $h$ and its Fourier transform.

\subsection{Proof of $\la_1 \notin [0, 0.25]$}
\label{subsec:subinterval1}
\begin{prop}\label{prop:subinterval1}
There is no eigenvalue in the interval $(0,1/4]$.
\end{prop}

\begin{proof}
We choose the test function
%~\cite[p. 185]{MR1325466}
\[
h(r)=\left(\frac{\sin r}{r}\right)^{4}.
\]
Then its Fourier transform $g$ is explicitly given by
\begin{align*}
g(u) = 
\frac{1}{192} \Big[ &(u-4)^3 \sgn(u-4)  - 4 (u-2)^3 \sgn(u-2) +6 u^3 \sgn(u)  \\
&- 4 (u+2)^3 \sgn(u+2) + (u+4)^3 \sgn(u+4) \Big].
\end{align*}
Note that $g$ is supported in the interval $[-4,4]$.
Given that $\sys(\X) > 5$, the hyperbolic classes do not contribute to $\calG$, resulting in $\calG = 0$.
Thus, the trace formula can be simply expressed as $\calS = \calI$.

Let us find an upper bound for $\calS$.
We put $\eps=10^{-3}$ and $L=1000$ and split the integral for $\calI$ as
\begin{align*}
\calI_1 & = 2(g(\X)-1)\int_0^{\eps} r h(r) \tanh(\pi r) \,dr\\
\calI_2 &= 2(g(\X)-1)\int_{\eps}^{L} r h(r) \tanh(\pi r) \,dr\\
\calI_3 & = 2(g(\X)-1)\int_{L}^{\infty}  r h(r) \tanh(\pi r) \,dr
\end{align*}
Using the inequality 
$\sin(x) \leq x - \frac{x^3}{6} + \frac{x^5}{120}$ for $x\geq 0$, 
we can check that
$h(r)\leq (1 - r^2/6 + r^4/120)^4 \leq 1-\frac{2 r^2}{3}+\frac{r^4}{5}$ for $r\leq 1$.
Also note that
$\tanh(\pi r) = \frac{e^{2 \pi  r}-1}{e^{2 \pi  r}+1} \leq 1$.
Thus,
\begin{equation}\label{eq:upper_bound_I1}
\calI_1 \leq 2(g(\X)-1) \int_0^{\eps} r \left(1-\frac{2 r^2}{3}+\frac{r^4}{5}\right)\,dr
< 6\times 10^{-6}.    
\end{equation}
For the second term, we employ interval arithmetic to obtain the following:
\[
\calI_2 = 2(g(\X)-1)\int_{\eps}^{L} r h(r) \tanh(\pi r) \,dr < 7.8684.
\]
Since the function $h(r)$ is bounded by $1/r^4$, 
\[
\calI_3 \leq 2(g(\X)-1) \int_{L}^\infty r \frac{1}{r^4} \,dr  = \frac{g-1}{L^2} =  6\times 10^{-6}.
\]
Note also that
$h(r_0) = h(i/2) = 16 \sinh\left(\frac{1}{2}\right)^4 > 1.1797$.
We can deduce that 
\begin{equation}\label{upper_bound_I}
\calS-h(r_0) = \calI-h(r_0) =  \calI_1+\calI_2+\calI_3-h(r_0)<6.7.    
\end{equation}
See the Sage~\cite{sagemath} notebook \texttt{subinterval1.ipynb} for the necessary numerical calculations.

Assume that $0\leq \la_1 \leq 1/4$.
Then $r_1 = i\nu_1$ for some $0\leq \nu_1 \leq 1/2$.
Because its multiplicity must be at least 7 (Lemma~\ref{prop:mult7}), $h(i\nu)=\frac{\left(e^{\nu }-e^{-\nu }\right)^4}{16 \nu ^4}\geq 1$ for $0\leq \nu \leq 1/2$, and $h(r)>0$ for all $r\in \R$, we have 
\[
\calS-h(r_0) = \sum_{n=1}^{\infty} h(r_n) \geq 7.
\]
This contradicts the bound \eqref{upper_bound_I} for $\calS$.
\end{proof}

\subsection{Proof of $\la_1 \notin [0.25, 1]$}
\label{subsec:subinterval2}
%Booker-Strombergsson, Numerical computations with the trace formula and the selberg eigenvalue conjecture, 148p
%$M=15$, $X=6$, $\delta = X/2M = 1/5$
Consider the function $f$ given by
\begin{equation}\label{eqn:hsqrth}
f(r)=\left(\frac{\sin(\delta r / 2)}{\delta r / 2}\right)^{2} \sum_{n=0}^{M-1} x_{n} \cos(\delta n r),
\end{equation}
where $M=15$ and $\delta = 1/5$.
The coefficients $x_i$ are given by:
\begin{align*}
x_{0} & = 0.10495, & x_{5} & = 0.16480, & x_{10} & = 0.06400, \\
x_{1} & = 0.20797, & x_{6} & = 0.14688, & x_{11} & = 0.04416, \\
x_{2} & = 0.20226, & x_{7} & = 0.12715, & x_{12} & = 0.02632, \\
x_{3} & = 0.19294, & x_{8} & = 0.10628, & x_{13} & = 0.01154, \\
x_{4} & = 0.18032, & x_{9} & = 0.08497, & x_{14} & = 0.00046.
\end{align*}

We select the test function $h$ to be the square of $f$:
\begin{align}
h(r) & = f(r)^{2} \label{eqn:hfsquare} \\
     & = \left(\frac{\sin(\delta r / 2)}{\delta r / 2}\right)^{4} \sum_{{i=0}}^{{M-1}} \sum_{{j=0}}^{{M-1}}     
     \cos(i \delta r) \cos(j \delta r) x_{i} x_{j} \nonumber \\
     & = \left(\frac{\sin(\delta r / 2)}{\delta r / 2}\right)^{4}\sum_{{i=0}}^{{M-1}} \sum_{{j=0}}^{{M-1}}     
     \frac{1}{2} \left[ \cos((i-j) \delta r) + \cos((i+j) \delta r) \right]x_{i} x_{j}. \nonumber
\end{align}

\begin{prop}\label{prop:upper_bound_II}
For the test function $h$ in \eqref{eqn:hfsquare}, we have $\calS-h(r_0)<7$.
\end{prop}

\begin{proof}
Let $h_{ij}(r) : =\frac{1}{2} \left(\frac{\sin(\delta r / 2)}{\delta r / 2}\right)^{4}
     \left[ \cos((i-j) \delta r) + \cos((i+j) \delta r) \right]$.
Then 
\[
h(r) = \sum_{{i=0}}^{{M-1}} \sum_{{j=0}}^{{M-1}}h_{ij}(r)x_i x_j.
\]

Then we obtain the following expression from \eqref{eqn:hfsquare}:
\begin{equation}\label{eq:calS}
\mathcal{S} = \sum_{n=0}^{\infty} h(r_n) = \sum_{i=0}^{M-1} \sum_{j=0}^{M-1} a_{ij} x_{i} x_{j},
\end{equation}
where $a_{ij} =  \sum_{n=0}^{\infty} h_{ij}(r_n)$.

For fixed $\delta$ and $j$, introduce the following function with respect to the variable $r$:
\[
h(r, \delta, j) = \left(\frac{\sin(\delta r / 2)}{\delta r / 2}\right)^{4}\cos(j \delta r)
\]
and its Fourier transform, which is a function of $u$:
\begin{align*}
g(u, \delta, j) = \frac{1}{24 \delta^4} (&6 (u - j \delta)^3 \text{sgn}(u - j \delta) + (u - 2\delta - j \delta)^3 \text{sgn}(u - 2\delta - j \delta) \\
&- 4 (u - \delta - j \delta)^3 \text{sgn}(u - \delta - j \delta) \\
&- 4 (u + \delta - j \delta)^3 \text{sgn}(u + \delta - j \delta) + (u + 2\delta - j \delta)^3 \text{sgn}(u + 2\delta - j \delta) \\
&+ 6 (u + j \delta)^3 \text{sgn}(u + j \delta) + (u - 2\delta + j \delta)^3 \text{sgn}(u - 2\delta + j \delta) \\
&- 4 (u - \delta + j \delta)^3 \text{sgn}(u - \delta + j \delta) - 4 (u + \delta + j \delta)^3 \text{sgn}(u + \delta + j \delta) \\
&+ (u + 2\delta + j \delta)^3 \text{sgn}(u + 2\delta + j \delta)).
\end{align*}
For each $0\leq i\leq 2M-2 = 28$, we put
\begin{align*}
\calI_i & = 2(g(\X)-1)\int_0^\infty r h(r , \delta, i) \tanh(\pi r) \,dr,\\
\calG_i & = \sum_{\{\gamma\} \in \calC(\Gamma)} \frac{\La(\gamma)}{2 \sinh(\ell(\gamma)/2)}g(\ell(\gamma), \delta, i).
\end{align*}
From the Selberg trace formula, we can write
\begin{equation}\label{eq:a_ij}
a_{ij} = \frac{1}{2}(\mathcal{I}_{|i-j|} + \mathcal{I}_{i+j} + \mathcal{G}_{|i-j|} + \mathcal{G}_{i+j}).
\end{equation}

Note that each 
$g(u, \delta, j)$ has support in $[-(2+j)\delta, (2+j)\delta]$.
For $0\leq j \leq 26$, $\calG_{j}=0$ because $\ell_0 =\sys(\X)$ is outside the support.
Specifically, $(2+j)\delta\leq 28\cdot \frac{1}{5} <\ell_0 \approx 5.796$.
For $j=27$ or $28$, only the conjugacy classes of the shortest length can contribute to the sum $\calG_j$.
That is to say,
\[
\calG_{j} = \frac{252\ell_0}{2 \sinh(\ell_0/2)}g(\ell_0, \delta, j),\, j\in \{27, 28\}.
\]

Let us consider $\calI_i = 2(g(\X)-1)\int_0^\infty r h(r , \delta, i) \tanh(\pi r) \,dr$.
Put $\eps=10^{-3}$ and $L=10^5$.
We intend to approximate $\calI_i$ by
\[
\calI_i' : = 2(g(\X)-1)\int_\eps^L r h(r , \delta, i) \tanh(\pi r) \,dr.
\]
Let us find a bound for the deviation between $\calI_i$ and $\calI_i'$.
As in \eqref{eq:upper_bound_I1}, we have
\[
\left|h(r, \delta, j)\right| \leq 1 - \frac{r^2 \delta^2}{6} + \frac{r^4 \delta^4}{80},\quad r\leq \frac{1}{\delta}.
\]
From this, it follows that:
\[
2(g(\X)-1)\left|\int_0^\eps r h(r , \delta, i) \tanh(\pi r) \,dr\right| \leq 
2(g(\X)-1)\int_0^\eps r(1 - \frac{r^2 \delta^2}{6} + \frac{r^4 \delta^4}{80})  \,dr
<6\times 10^{-6}.
\]
And using $|h(r,\delta, i)|\leq 1/r^4$, we obtain
\[
2(g(\X)-1)\left|\int_L^\infty r h(r , \delta, i) \tanh(\pi r) \,dr\right| \leq 
2(g(\X)-1)\int_L^\infty r(\frac{1}{\delta r/2})^4  \,dr
=6\times 10^{-6}.
\]
Therefore,
\[
\left|\calI_i - \calI_i' \right|<1.2\times 10^{-5}.
\]
In this way, we can estimate the values for $\calI_0,\dots, \calI_{28}$, $\calG_{27}$, and $\calG_{28}$ within a known error bound by utilizing interval arithmetic as detailed below:
\[
\begin{array}{c|c|c|c|c|c}
& \text{approximate value} & \text{error bound} & & \text{approximate value} & \text{error bound} \\
\hline
\calI_{0} & 831.2770 & 0.0000383 & \calI_{1} & -180.9242 & 0.0000419 \\
\calI_{2} & -110.5626 & 0.0000416 & \calI_{3} & -38.3114 & 0.0000548 \\
\calI_{4} & -20.5191 & 0.0000292 & \calI_{5} & -12.9689 & 0.0000616 \\
\calI_{6} & -9.0154 & 0.0000125 & \calI_{7} & -6.6722 & 0.0000182 \\
\calI_{8} & -5.1618 & 0.0000379 & \calI_{9} & -4.1264 & 0.0000124 \\
\calI_{10} & -3.3822 & 0.0000410 & \calI_{11} & -2.8265 & 0.0000515 \\
\calI_{12} & -2.3986 & 0.0000387 & \calI_{13} & -2.0604 & 0.0000594 \\
\calI_{14} & -1.7873 & 0.0000287 & \calI_{15} & -1.5626 & 0.0000498 \\
\calI_{16} & -1.3749 & 0.0000347 & \calI_{17} & -1.2160 & 0.0000273 \\
\calI_{18} & -1.0800 & 0.0000193 & \calI_{19} & -0.9625 & 0.0000470 \\
\calI_{20} & -0.8601 & 0.0000281 & \calI_{21} & -0.7704 & 0.0000150 \\
\calI_{22} & -0.6913 & 0.0000281 & \calI_{23} & -0.6213 & 0.0000201 \\
\calI_{24} & -0.5591 & 0.0000605 & \calI_{25} & -0.5036 & 0.0000570 \\
\calI_{26} & -0.4539 & 0.0000547 & \calI_{27} & -0.4095 & 0.0000170 \\
\calI_{28} & -0.3696 & 0.0000159 & & & \\
\hline
\calG_{27} & 0.0002 & 0.0000414 & \calG_{28} & 35.5542 & 0.0000587 \\
\end{array}
\]
See the Sage notebook \texttt{subinterval2.ipynb} for the necessary numerical calculations.

By applying interval arithmetic to \eqref{eq:calS} and \eqref{eq:a_ij}, we finally obtain
\[
\calS- h(r_0) < 6.76,
\]
thereby establishing the proposition.
\end{proof}

\begin{lemma}\label{lem:subinterval2}
For $r\in [0,0.87]$, $h(r)\geq 1$.    
\end{lemma}

\begin{proof}
For each $0\leq i\leq 14$,
we have $x_i>0$ and
the functions
$\left(\frac{\sin(\delta r / 2)}{\delta r / 2}\right)^{2}$ 
and $\cos(i\delta r)$ are monotonically decreasing over the given interval since
$i\delta = i/5 < \pi$.
Thus, $f(r)$ in \eqref{eqn:hsqrth} is monotonically decreasing over $[0,0.87]$.
Using interval arithmetic, we can check $f(0.87)\geq 1$;
see the Sage notebook \texttt{subinterval2.ipynb} for the evaluation of $f(0.87)$.
This implies $f(r)\geq 1$ for $r\in [0,0.87]$.
\end{proof}

\begin{prop}\label{prop:subinterval2}
There is no eigenvalue in the interval $[1/4,1]$.
\end{prop}
\begin{proof}
From \refprop{prop:subinterval1}, we know $\la_1\geq 1/4$.
Therefore, $r_n\in \R$ for all $n\geq 1$.
Assume that $0\leq r_1 \leq 0.87$.
Then we have $\calS-h(r_0)=\sum_{n=1}^{\infty} h(r_n)\geq 7$ because the multiplicity of $\la_1$ is at least 7 (Lemma~\ref{prop:mult7}), $h(r_1)\geq 1$ (Lemma \ref{lem:subinterval2}), and $h(r)=f(r)^2\geq 0$ for all $r\in \R$ \eqref{eqn:hfsquare}.
This contradicts Proposition~\ref{prop:upper_bound_II}.
%Therefore, $r_1 > 0.87$ and $\la_1> r_1^2+1/4=1.0069>1$.
Therefore, $r_1 > 0.87$ and $\la_1 = r_1^2+1/4>1.0069>1$.
\end{proof}

\subsection{Proof of $\la_1\notin [1,1.23]$}
\label{subsec:subinterval3}
\begin{prop}\label{prop:subinterval3}
There is no eigenvalue in the interval $[1,1.23]$.
\end{prop}

\begin{proof}
Let $a = 1$ and $b=1.23$.
We define the test functions $h,g$ as in \eqref{eq:gh_hermite}.
We choose the sequence of rational numbers \(\{s_n\}_{n=0}^{17}\) such that the two polynomials
\begin{align*}
u(y) &= \sum_{n=0}^{17} s_n L_n^{(-1/2)}(y),\\
v(y) &= \sum_{n=0}^{17} (-1)^n s_n L_n^{(-1/2)}(y),
\end{align*}
meet the following conditions:
\begin{itemize}
    \item \(u(y)\) has a simple root at \(33\) and a double root at \(1937\).
    \item \(v(y)\) possesses double roots at 1.83, 4.94, 9.86, 15.35, 22.16, 32.22, 46.41 and additionally fulfills the condition \(v(a-1/4) = 1\).
\end{itemize}
The above conditions produce a system of linear equations for the coefficients \(\{s_n\}_{n=0}^{17}\).
This system has a unique solution over the rationals.

For the necessary numerical computations in this proof, see the Sage notebook \texttt{subinterval3.ipynb}.
We can confirm that both $u$ and $v$ have roots with the specified orders at the designated locations.

Recall that $\ff(x) = u(x^2) e^{-x^2/2}$ and $\hatf(\xi)=v(\xi^2) e^{-\xi^2/2}$.
We need to ensure that they satisfy the following conditions:
\begin{enumerate}
    \item $\ff(x)\leq 0$ for every $x \geq \sys(\X)$.
    \item $\hatf(r(\la))\geq c=0.32388$ for every $\la \in [a, b]$.
    \item $\hatf(r(\la))\geq 0$ for every $\la \geq b$.
\end{enumerate}
To count the number of zeros of a polynomial over a given interval, we can employ Sturm's theorem.
This has been implemented in \texttt{PARI/GP}~\cite{PARI2}, and is also accessible through \texttt{Sage}.

(1) 
We can count the number of zeros of $u$ in $[33, \infty)$ to make sure that it does not change sign, and verify that $u''$ is negative at its unique double zero so that $u$ and hence $\ff$ is non-positive on $[\sqrt{33},\infty)$.
Because $\sqrt{33}<5.75<\sys(\X)$, $\ff(x)\leq 0$ for $x \geq \sys(\X)$.

(2)
We verify that \(2v'-2v\) does not have a root between \(a-1/4\) and \(b-1/4\).
This implies that \(\hatf\) does not exhibit a critical point within the interval from \(r(a)\) to \(r(b)\) since the derivative of \(\hatf\) is given by:
\[
\hatf'(\xi) = (2v'(\xi^2)-v(\xi^2)) e^{-\xi^2/2} \xi.
\]
This implies that the minimum of $\hatf$ on the interval $[r(a),r(b)]$ is achieved at the endpoints. 
Using interval arithmetic, we can check that $\hatf(r(\la)) \geq  c$ for $\la\in [a, b]$.
Note that 
$\hatf(r(a)) \approx 0.6872893$ and $\hatf(r(b)) \approx 0.3238802$.

(3)
To demonstrate that \(\hatf(r(\la)) \geq 0\) holds for all \(\la \geq b\), it's necessary to confirm that \(\hatf\) retains a consistent sign over the interval \([r(b), \infty)\).
To this end, we apply Sturm's theorem to verify that \(v\) has precisely \(7\) distinct zeros, all of which are double zeros, in the interval \([b-1/4,\infty)\). 
Consequently, we deduce that \(v\) does not change sign over \([b-1/4,\infty)\) and, hence, \(\hatf\geq 0 \) over the interval \([r(b),\infty)\) because \(\hatf(r(b)) \geq c\).

With the above test functions $h,g$, we will establish the following inequalities:
\[
\mathcal{I} \leq 9.197, \quad \mathcal{G} \leq -0.10356, \quad -\hatf(r_0) \leq -6.854.
\]

To find an upper bound for $\calI$,
we separate the integral into two parts $\calI = \calI_1+\calI_2$ :
\begin{align*}
\calI_1 & = 2(g(\X)-1) \int_0^{100} r \hatf(r) \tanh(\pi r) dr,\\
\calI_2 & = 2(g(\X)-1) \int_{100}^\infty r \hatf(r) \tanh(\pi r) dr.
\end{align*}
Using interval arithmetic, we can find
\[
\calI_1 \leq 9.1971.
\]
For the remaining part, we have
\[
\calI_2 = 2(g(\X)-1) \int_{100}^\infty r \hatf(r) \tanh(\pi r) dr\leq 2(g(\X)-1) \int_{100}^\infty r \hatf(r)dr.
\]
Observe that for any odd integer $k\geq 1$, the function $r^k e^{-\frac{r^2}{2}}$ has an antiderivative of the form of $e^{-\frac{r^2}{2}}$ multiplied by a polynomial in $r$. 
Consequently, the function $r \hatf(r) = r \, v(r^2) e^{-r^2/2}$ has an antiderivative of the form $V(r) = p(r)e^{-r^2/2}$, where $p$ is a polynomial.
By evaluating it, we obtain the inequality
\[
\calI_2 \leq -2(g(\X)-1)V(100) \approx 1.747\times 10^{-2124}.
\]
Thus, $\calI_2$ is so small as to be practically negligible in comparison to $\calI_1$.

Now we turn to $\calG$.
Because $\ff(x)\leq 0$ for every $x \geq \sys(\X)$,
\begin{align*}
\calG & = \frac{1}{\sqrt{2\pi}}\sum_{\{\gamma\} \in \calC(\Gamma)} \frac{\La(\gamma)}{2 \sinh(\ell(\gamma)/2)} \ff(\ell(\gamma)) \\
& \leq \frac{1}{\sqrt{2\pi}}\sum_{\{\gamma\} \in P} \frac{\ell(\gamma)}{2\sinh(\ell(\gamma)/2)} \ff(\ell(\gamma)),
\end{align*}
where $P$ is any subset of $\calC(\Gamma)$.
By employing interval arithmetic to sum over conjugacy classes with the two shortest lengths described in \refprop{prop:shortgeo}, we obtain
\[
\calG \leq -0.10356.
\]

Finally, we can verify that $\hatf(r_0)\geq 6.854$ using interval arithmetic.

Now we put these pieces together.
Because $\hatf(r(\la))\geq 0$ for $\la \geq b$, the Selberg trace formula gives
\[
\calI+\calG = \calS  =  \sum_{n=1}^{\infty}\hatf(r_n)
\geq \hatf(r_0)+ \sum_{\substack{n=1 \\ \la_n\in [a,b]}}^{\infty}
\hatf(r_n).
\]
From $\hatf(r(\la))\geq c$ for $\la\in [a,b]$, 
we obtain
\[
\calI+\calG - \hatf(r_0)  \geq  c m(a,b),
\]
where $m(a,b)$ denote the number of eigenvalues of the Laplacian in the interval $[a,b]$ counting multiplicity.
Therefore, $m(a,b)$ is less than or equal to
\[
\frac{1}{c}\left(\calI+\calG - \hatf(r_0) \right) \leq 6.9148<7.
\]
From Lemma~\ref{prop:mult7}, we can finish the proof of the proposition.
\end{proof}

\subsection{Proof of $\la_1\in [1.23, 1.26]$ and determination of its multiplicity}\label{subsec:subinterval4}
\begin{prop}\label{prop:subinterval4}
There is an eigenvalue in the interval $[1.23, 1.26]$.
\end{prop}

\begin{proof}
Let $a = 1.23$ and $b = 1.26$.
We define the test functions $h,g$ as in \eqref{eq:gh_hermite}.
For this, we choose the sequence of rational numbers \(\{s_n\}_{n=0}^{19}\) such that the two polynomials
\begin{align*}
u(y) &= \sum_{n=0}^{19} s_n L_n^{(-1/2)}(y),\\
v(y) &= \sum_{n=0}^{19} (-1)^n s_n L_n^{(-1/2)}(y),
\end{align*}
meet the following conditions:

\begin{itemize}
    \item \(u(y)\) has a double root at \(4190\).
    \item \(v(y)\) possesses double roots at 1.87, 4.89, 9.46, 13.81, 19.97, 28.86, 39.99, 54.93 and additionally fulfills the conditions \(v(0)=1\) and \(v(b-1/4) = -1/10^5\).
\end{itemize}
There is a unique choice of \(\{s_n\}_{n=0}^{17}\).

See the Sage notebook \texttt{subinterval4-1.ipynb} for the necessary numerical computations below.
We can easily confirm that $u$ and $v$ have roots with the specified multiplicities at the designated locations.

We need to ensure that $h$ and $g$ satisfy the following conditions:
\begin{enumerate}
    \item $\ff(x)\geq 0$ for every $x \geq 0$.
    \item $\hatf(r(\la))\leq c=0.00141$ for every $\la \in [a, b]$.
    \item $\hatf(r(\la))\leq 0$ for every $\la \geq b$.
\end{enumerate}
Again, we apply Sturm's theorem to count the number of roots of a polynomial over a given interval.

(1) 
We can check that $u(0)> 0$ and $u$ has a unique root at \(4190\) in the interval $[0, \infty)$, which is a double root.
By confirming that $u''(4190)$ is greater than zero, we can establish that both $u$ and $\ff$ maintain non-negativity throughout the interval $[0,\infty)$.

(2)
We verify that \(2v'-2v\) does not have a root in \([a-1/4, b-1/4]\).
This implies that the maximum of $\hatf$ on the interval $[r(a),r(b)]$ is achieved at the endpoints. 
Using interval arithmetic, we can check that $\hatf(r(\la)) \leq  0.00141$ for $\la\in [a, b]$.
Note that 
$\hatf(r(a)) \approx 0.00140988$ and $\hatf(r(b)) \approx -6.03506\times 10^{-6}$.

(3)
We verify that \(v\) has precisely \(8\) distinct zeros, all of which are double zeros, in the interval \([b-1/4,\infty)\).
We can confirm that $v''<0$ at all those roots, and $v(b-1/4)<0$.
Consequently, we deduce that \(v\) does not change sign over \([b-1/4,\infty)\) and, hence, \(\hatf\leq 0 \) over the interval \([r(b),\infty)\).

With the above test functions $h,g$, we will establish the following inequalities using interval arithmeti:
\[
\mathcal{I} \geq 1.08088, \quad \mathcal{G} \geq 1.25565, \quad -\hatf(r_0) \geq -2.33608.
\]

To find a bound for $\calI$,
we separate the integral into two parts $\calI = \calI_1+\calI_2$ :
\begin{align*}
\calI_1 & = 2(g(\X)-1) \int_0^{100} r \hatf(r) \tanh(\pi r) dr,\\
\calI_2 & = 2(g(\X)-1) \int_{100}^\infty r \hatf(r) \tanh(\pi r) dr.
\end{align*}
Using interval arithmetic, we can find
\[
\calI_1 \geq 1.08088.
\]
For the remaining part, we have
\[
\calI_2 = 2(g(\X)-1) \int_{100}^\infty r \hatf(r) \tanh(\pi r) dr\leq 2(g(\X)-1) \int_{100}^\infty r \hatf(r)dr,
\]
which is negligible in comparison to $\calI_1$ as in the proof of \refprop{prop:subinterval3}.

Now we consider $\calG$.
Because $\ff(x)\geq 0$ for every $x \geq 0$,
\[
\calG \geq \frac{1}{\sqrt{2\pi}}\sum_{\{\gamma\} \in P} \frac{\ell(\gamma)}{2\sinh(\ell(\gamma)/2)} \ff(\ell(\gamma)),
\]
where $P$ is any subset of $\calC(\Gamma)$.
By employing interval arithmetic to sum over conjugacy classes with the two shortest lengths described in \refprop{prop:subinterval3}, we obtain
\[
\calG \geq 1.25565.
\]

Finally, we can check that $\hatf(r_0)\leq 2.33608$ using interval arithmetic.

Because $\hatf(r(\la))\leq 0$ for $\la \geq b$, the Selberg trace formula gives
\[
\calI+\calG = \calS = \sum_{n=1}^{\infty}\hatf(r_n)
\leq 
\hatf(r_0)+ \sum_{\substack{n=1 \\ \la_n\in [a,b]}}^{\infty}
\hatf(r_n).
\]
Let $m(a,b)$ be the number of eigenvalues of the Laplacian in the interval $[a,b]$ counting multiplicity.
From $\hatf(r(\la))\leq c$ for $\la\in [a,b]$, 
we obtain
\[
\calI+\calG - \hatf(r_0)  \leq  c m(a,b),
\]
which implies
\begin{equation}
m(a,b) \geq \frac{1}{c}\left(\calI+\calG - \hatf(r_0) \right) \geq 0.3237.
\end{equation}
This proves the proposition.
\end{proof}

\begin{prop}\label{prop:subinterval4-2}
The number of eigenvalues in the interval $[1.23,1.26]$ counting multiplicity is 7.
\end{prop}
\begin{proof}
Let $a = 1.23$ and $b=1.26$. 
We define the test functions $h,g$ as in \eqref{eq:gh_hermite}.
We choose the sequence of rational numbers \(\{s_n\}_{n=0}^{17}\) such that the two polynomials
\begin{align*}
u(y) &= \sum_{n=0}^{17} s_n L_n^{(-1/2)}(y),\\
v(y) &= \sum_{n=0}^{17} (-1)^n s_n L_n^{(-1/2)}(y),
\end{align*}
meet the following conditions:
\begin{itemize}
    \item \(u(y)\) has a simple root at \(33\) and a double root at \(624\).
    \item \(v(y)\) possesses double roots at 1.91, 4.97, 9.92, 15.56, 22.31, 32.18, 46.81 and additionally fulfills the condition \(v(a-1/4) = 1\).
\end{itemize}
Again, there is a unique choice of \(\{s_n\}_{n=0}^{17}\).

See the Sage notebook \texttt{subinterval4-2.ipynb} for the necessary numerical computations below.
We can easily confirm that $u$ and $v$ have roots with the specified multiplicities at the designated locations.

We need to ensure that they satisfy the following conditions:
\begin{enumerate}
    \item $\ff(x)\leq 0$ for every $x \geq \sys(\X)$.
    \item $\hatf(r(\la))\geq c=0.55299$ for every $\la \in [a, b]$.
    \item $\hatf(r(\la))\geq 0$ for every $\la \geq b$.
\end{enumerate}
We can verify these statements similarly to the method described in the proof of \refprop{prop:subinterval3}.
Also we have the following inequalities:
\[
\mathcal{I} \leq 16.242, \quad \mathcal{G} \leq -0.162, \quad -\hatf(r_0) \leq -11.826.
\]

Because $\hatf(r(\la))\geq 0$ for $\la \geq b$, the Selberg trace formula gives
\[
\calI+\calG = \calS  =  \sum_{n=1}^{\infty}\hatf(r_n) \geq
\hatf(r_0) + \sum_{\substack{n=1 \\ \la_n\in [a,b]}}^{\infty}\hatf(r_n).
\]
From $\hatf(r(\la))\geq c$ for $\la\in [a,b]$, 
we obtain
\[
\calI+\calG - \hatf(r_0)  \geq  c m(a,b)
\]
where $m(a,b)$ denote the number of eigenvalues of the Laplacian in the interval $[a,b]$ counting multiplicity.
Therefore, $m(a,b)$ is less than or equal to
\begin{equation}
\frac{1}{c}\left(\calI+\calG - \hatf(r_0) \right) \leq \frac{1}{0.553}(16.242 - 0.162 - 11.826)< 7.70.
\end{equation}
This implies that the multiplicity is less than 8.
From \refprop{prop:subinterval4}, we know the multiplicity is positive.
By Lemma~\ref{prop:mult7}, it must be 7, which proves the proposition.
\end{proof}

\section{Numerical determination of the eigenspace associated to $\la_1$}
Having established that the first eigenvalue of the Fricke-Macbeath surface has a multiplicity of 7, our objective in this section is to numerically determine which 7-dimensional representation of $\Aut(\X)=\Delp/\Gamma$ corresponds to the eigenspace associated to $\la_1$.
There are four 7-dimensional irreducible representations of $\Aut(\X)\cong \PSL_2(8)$.
Refer to Table~\ref{table:character_table} for the character table of $\Aut(\X)$.
You can find the GAP code in Appendix~\ref{sec:appB} to obtain the character table.
See~\cite[p. 6]{MR0827219} for more information about the character table.
In the character table, \( A, B, C, D, E, \) and \( F \) are defined as follows:
\begin{equation} \label{eq:abcdef}
\begin{aligned}
A &= \zeta_7^3 + \zeta_7^4, & B &= \zeta_7^2 + \zeta_7^5, & C &= \zeta_7 + \zeta_7^6, \\
D &= -\zeta_9^4 - \zeta_9^5, & E &= -\zeta_9^2 - \zeta_9^7, \\
F &= \zeta_9^2 + \zeta_9^4 + \zeta_9^5 + \zeta_9^7.
\end{aligned}
\end{equation}
Here, $\zeta_n = e^{2\pi \i /n}$.

\begin{table}[h]
    \centering
    \begin{tabular}{l|ccccccccc}
          & \textbf{1A} & \textbf{2A} & \textbf{3A} & \textbf{7A} & \textbf{7B} & \textbf{7C} & \textbf{9A} & \textbf{9B} & \textbf{9C} \\
         %$#c$ & 1 & 63 & 56 & 72 & 72 & 72 & 56 & 56 & 56 \\  
         %#centralizer & 504 & 8 & 9 & 7 & 7 & 7 & 9 & 9 & 9 \\
         \hline
        \(\chi_1\) & 1 & 1 & 1 & 1 & 1 & 1 & 1 & 1 & 1 \\
        \(\chi_2\) & 7 & -1 & -2 & 0 & 0 & 0 & 1 & 1 & 1 \\
        \(\chi_3\) & 7 & -1 & 1 & 0 & 0 & 0 & \(D\) & \(F\) & \(E\) \\
        \(\chi_4\) & 7 & -1 & 1 & 0 & 0 & 0 & \(E\) & \(D\) & \(F\) \\
        \(\chi_5\) & 7 & -1 & 1 & 0 & 0 & 0 & \(F\) & \(E\) & \(D\) \\
        \(\chi_6\) & 8 & 0 & -1 & 1 & 1 & 1 & -1 & -1 & -1 \\
        \(\chi_7\) & 9 & 1 & 0 & \(A\) & \(C\) & \(B\) & 0 & 0 & 0 \\
        \(\chi_8\) & 9 & 1 & 0 & \(B\) & \(A\) & \(C\) & 0 & 0 & 0 \\
        \(\chi_9\) & 9 & 1 & 0 & \(C\) & \(B\) & \(A\) & 0 & 0 & 0
    \end{tabular}
    \caption{Character table of \(\Aut(\X) \cong \PSL_2(8)\); the values of \( A, B, C, D, E, \) and \( F \) are given in \eqref{eq:abcdef}}
    \label{table:character_table}
\end{table}

GAP enumerates representative elements from each conjugacy class as follows:
\begin{enumerate}
\item[\textbf{1A.}] $\id$
\item[\textbf{2A.}] \( a^2 b (b a^2)^2 b^3 a^2 b \)
\item[\textbf{3A.}] \( a \)
\item[\textbf{7A.}] \( a^2 b (b a^2)^2 b^3 a^2 \)
\item[\textbf{7B.}] \( b^4 a^2 (b a^2 b)^2 \)
\item[\textbf{7C.}] \( a^2 b^3 (b a^2)^2 b^2 \)
\item[\textbf{9A.}] \( (a b^2 a)^2 a b^3 a \)
\item[\textbf{9B.}] \( a (b a^2 b)^2 b^2 a^2 \)
\item[\textbf{9C.}] \( a b^2 (b^2 a^2)^2 b \)
\end{enumerate}
For the actions of the generators \( a \) and \( b \) on \( \disk \), recall \eqref{eq:matrixA} and \eqref{eq:matrixB}.

It is noteworthy that the outer automorphism group of \( \PSL_2(8) \) has order 3.
An outer automorphism interchanges the conjugacy classes of elements of order 9 with each other and similarly for those of order 7.
Consequently, there is an inherent ambiguity in assigning specific labels like \textbf{9A}, \textbf{9B}, \textbf{9C} to elements of order 9, and similarly, \textbf{7A}, \textbf{7B}, \textbf{7C} to those of order 7.
Throughout this section, we will adhere to the aforementioned convention for naming conjugacy classes.

Here we present a numerical method to compute the character of a representation obtained from a set of eigenfunctions. 
Assume that we have a set of eigenfunctions, denoted as \(L = \{f_1,\dots, f_n\}\), which have close eigenvalues. 
These eigenfunctions are expected to form a basis for a representation of \(\Aut(\X)\), which may not necessarily be irreducible.

One of the core assumptions in our approach is that we can approximate the value \(f_i(z)\) for a given point \(z \in \F\).
For instance, by employing the finite element method on \(\F\), along with periodic boundary conditions determined by side-pairing transformations, we can achieve such approximations for eigenfunctions of hyperbolic surfaces.
The application of these methods to the Bolza surface, a genus 2 Riemann surface with the highest number of symmetries, has been introduced in~\cite{MR1028714}.
Cook's dissertation~\cite{cook2018} provides FreeFEM++~\cite{MR3043640} code to compute low-lying eigenvalues and eigenfunctions for certain compact Riemann surfaces of low genera, albeit not including the Fricke-Macbeath surface.
For further details, see appendices C, D, and E in~\cite{cook2018}.
Meanwhile, an alternative approach detailed in~\cite{MR3009726} delivers exceptionally accurate numerical eigenvalues and eigenfunctions.
Although this method promises accuracy, its implementation seems to be quite challenging.

For our numerical evaluations, we have utilized Mathematica~\cite{Mathematica}.
Within this environment, the \texttt{NDEigensystem} function serves as an instrumental tool for determining the eigenvalues and eigenfunctions of differential operators over a given region.
A distinctive capability of this function is its adeptness in handling periodic boundary conditions.
With the \texttt{PeriodicBoundaryCondition} option, the function allows specific boundary pairs to be treated as if they were identical.

To begin, we initiate the following preliminary procedures:
\begin{itemize}
    \item Select an interior point \(z_0\) within \(\F\) that is not fixed by any element of \(\Del\). Additionally, ensure that the orbits of \(z_0\) under \(\Del\) do not intersect the boundary of \(\F\). 
    An appropriate point can readily be located within the fundamental triangle of \(\Del\), as shown in Figure~\ref{fig:237_tri}.
    \item Determine representatives \(S=\{M_1,\dots, M_d\}\subseteq \Delp\) for 
    \(\Aut(\X) = \Delp/\Gamma\)
    such that each \(z_k : = M_k z_0\) lies within \(\F\) for all $k\in \{1,\dots, d\}$. Here, $d=|\Aut(\X)|$. 
    We can find such a list, for example, as follows :
    \begin{itemize}
        \item Enumerate elements of the Coxeter group \(\Del\) based on length.
        \item Collect all elements \(M\in \Del\) of even length such \(M z_0\) is contained in the interior of \(\F\).
        \item Given our choice of \(z_0\), these collected elements serve as representatives for \(\Aut(\X) = \Delp/\Gamma\).
    \end{itemize}
    \item For each eigenfunction \(f\) in \(L\), compute the approximate values of \(f(z_k)\) for all \(k\in \{1,\dots, d\}\).
\end{itemize}

Given $M_l\in S$ and $f_i \in L$,  we want to express the transformed function $M_l\cdot f_i$ as a linear combination of $\{f_1,\dots, f_n\}$, meaning that
\[    
(M_l\cdot f_i)(z) = f_i(M_l^{-1}z) = \sum_{j=1}^n \rho(M_l,i,j)f_j(z),\, z\in \F.
\]    
Here, the matrix $\left(\rho(M_l,i,j)\right)_{i,j}$ serves as the representation matrix for $M_l$.
Our goal is to approximate these matrix entries, treating $\{\rho(M_l,i,j)\}$ as the unknowns we want to determine.

For a given pair $(l,i)$, the equation
\begin{equation}
\sum_{j=1}^n \rho(M_l,i,j)f_j(z_k) = f_i(M_l^{-1}z_k), \, k=1,\dots, d
\label{eq:rho_equation}
\end{equation}
yields $d$ relations with $n$ unknowns $\{\rho(M_l,i,j) : j=1, \dots, n\}$.
Therefore, when $d$ significantly exceeds $n$, we obtain an overdetermined system of linear equations for 
those unknowns.
For our particular case, the values are $d=504$ and $n=7$.
To handle such overdetermined systems, a common approach is to use the least squares method.
This method seeks to minimize the sum of the squares of the residuals, which are the differences between the left and the right sides of \eqref{eq:rho_equation}.
By applying this method, we can approximate the values of $\{\rho(M_l,i,j)\}$.

To enhance the system's robustness further, one could perturb the initial point $z_0$ to obtain $z_0'$ and then consider its orbits $z_k'$ under the action of $\Aut(\X)$. 
This leads to an extended set of relations, in addition to those in \eqref{eq:rho_equation}:
\[
\sum_{j=1}^n \rho(M_l,i,j)f_j(z_k') = f_i(M_l^{-1}z_k'), \, k=1,\dots, d.
\]
By doing so, we can also increase the number of equations, even when $d$ is small compared to $n$.

Note that on the right-hand side of \eqref{eq:rho_equation}, $M_l^{-1}z_k=M_l^{-1}M_k z_0$ may not be a point in $\F$.
Since our approach relies on knowing approximate values of $f_i(z_j)$ for orbits $z_j\in \F$ of $z_0$, calculating $f_i(M_l^{-1}z_k)$ presents a challenge.
Given $(l, k)$, we need to find $p\in \{1,\dots, d\}$ and $\gamma\in \Gamma$ such that
\begin{equation}
\gamma M_l^{-1}M_k = M_p.
\label{eq:findp}
\end{equation}
Once we find $(p,\gamma)$, we can simply set $f_i(M_l^{-1}z_k)=f_i(M_p z_0) = f_i(z_p)$ due to the $\Gamma$-invariance of $f_i$.

Equation \eqref{eq:findp} can be equivalently stated as:
\[
\gamma M_l^{-1}M_kz_0 = M_p z_0.
\]
This formulation is essentially about moving the point $M_l^{-1}M_kz_0 = M_l^{-1}z_k$ into the fundamental domain $\F$ through the action of $\Gamma$.
Notably, $\F$ is also the Dirichlet domain $D(0)$ from Theorem~\ref{thm:dirichlet}.
In~\cite{MR2541438}, Voight explored various algorithms related to such fundamental domains for Fuchsian groups.
Of particular interest is an algorithm that, when given a point \(z \in \disk\), returns a point \(z' \in \F\) and an element \(\gamma \in \Gamma\) such that \(z' = \gamma(z)\). 
This algorithm is detailed further in~\cite[Algorithm 4.3]{MR2541438}.
In the context of our study, we begin by setting \(z = M_l^{-1}z_k\).
Subsequently, we identify a side-pairing transformation \(\gamma_i\) such that \(\gamma_i z\) is closer to the fundamental polygon \(\F\) than the original location of \(z\), meeting the condition:
\[
\distD(0, \gamma_i z) \leq \distD(0, z).
\]
As long as the point $z$ is not located inside the fundamental domain $\F$, such a transformation \(\gamma_i\) exists.
This procedure can be iterated until we identify a \(\Gamma\)-orbit of \(z\) inside \(\F\).
By employing this method, we can effectively relocate $z$ to $\F$ using the generators $\{\gamma_i\}$ of $\Gamma$.
While determining \(\gamma\) and \(M_p\) for \eqref{eq:findp} may involve numerical calculations, we can verify its correctness in a symbolic manner using \eqref{eq:coxrelation}.

Once we apply the method described above to determine the 7-dimensional representation $V_{\la_1}$ formed by eigenfunctions with eigenvalue $\la_1$, we obtain the following numerical values for its character $\chi_{\la_1}$:
\[
\begin{array}{c|c|c|c|c|c|c|c|c}
\textbf{1A} & \textbf{2A} & \textbf{3A} & \textbf{7A} & \textbf{7B} & \textbf{7C} & \textbf{9A} & \textbf{9B} & \textbf{9C} \\
\hline
7.0000 & -1.0163 & 0.9614 & 0.0778 & 0.0231 & 0.0059 & -1.4985 & -0.3453 & 1.8829 \\
\end{array}
\]
Recall that the inner product of two characters $\chi, \chi'$ of $G$ are defined by
\[ \langle \chi, \chi' \rangle = \frac{1}{|G|} \sum_{g \in G} \chi(g) \ol{\chi'(g)}. \]
When we compute the inner product of characters $\chi_{\la_1}$ and $\chi_i$, we obtain the following approximate values of $m_i = \langle \chi_{\la_1}, \chi_i \rangle$ for each $i$:
\[
\begin{array}{c|c|c|c|c|c|c|c|c|c}
\hline
i & 1 & 2 & 3 & 4 & 5 & 6 & 7 & 8 & 9 \\
\hline
m_i & 0.0133 & 0.0150 & 0.0043 & -0.0037 & 0.9927 & 0.0152 & -0.0183 & -0.0119 & 0.0088 \\
\hline
\end{array}
\]
Note that $m_i$ is the multiplicity of $\chi_i$ in the decomposition of $\chi_{\la_1}$ into irreducible representations.
Given that these values are close approximations of integers, it becomes straightforward to guess their exact integer values.
Hence we propose the following conjecture:
\begin{conjecture}\label{conj:rep7}
Let $V_{\la_1}$ be the 7-dimensional representation of $\Aut(\X)$ formed by eigenfunctions with eigenvalue $\la_1$. 
The character of $V_{\la_1}$ is equal to $\chi_5$.
\end{conjecture}

From our numerical computations, we have obtained the following 7 approximate values for $\la_1, \dots, \la_7$:
\[1.25065,1.25336,1.25336,1.26103,1.26103,1.28662,1.38342.
\]
Despite the lack of precision in these values, it is evident that we can still obtain convincing evidence for the true values of $\chi_{\la_1}$.
Similarly, we have obtained the following 9 approximate values for $\la_8, \dots, \la_{16}$:
\[
2.05821,2.0602,2.06123,2.06488,2.0661,2.08243,2.09857,2.09857,2.2991.
\]
While the accuracy of these values remains uncertain, the eigenfunctions for these eigenvalues appear to span a 9-dimensional representation of $\Aut(\X)$.
Its character seems to be a good approximation to $\chi_9$ in Table~\ref{table:character_table}.

\section*{Acknowledgements}
This work was supported by a KIAS Individual Grant (SP067302) via the June E Huh Center for Mathematical Challenges at Korea Institute for Advanced Study.

% \section*{Conflict of Interest Statement}
% We state that there is no conflict of interest.

\appendix
\renewcommand{\thesection}{\Alph{section}}

\section{Word processing in $\Delta(2,3,7)$}\label{sec:appA}
The following Mathematica code defines a function named \texttt{reduce}.
For a given input, the function repeatedly employs a collection of transformation rules to a list, simplifying it according to the rules given in \eqref{eq:coxrelation} and the relations $s_i^2=1$.
%For instance, the initial three lines corresponds to the relations $s_i^2=1$.
\begin{lstlisting}[frame=single]
reduce[{a___,1,1,b___}]:=reduce[{a,b}]
reduce[{a___,2,2,b___}]:=reduce[{a,b}]
reduce[{a___,3,3,b___}]:=reduce[{a,b}]
reduce[{a___,3,1,b___}]:=reduce[{a,1,3,b}]
reduce[{a___,3,2,3,b___}]:=reduce[{a,2,3,2,b}]
reduce[{a___,3,2,1,3,b___}]:=reduce[{a,2,3,2,1,b}]
reduce[{a___,3,2,1,2,3,2,b___}]:=reduce[{a,2,3,2,1,2,3,b}]
reduce[{a___,2,1,2,1,2,1,2,b___}]:=reduce[{a,1,2,1,2,1,2,1,b}]
reduce[{a___,3,2,1,2,1,3,2,1,2,1,2,1,b___}]:=reduce[{a,2,3,2,1,2,1,3,2,1,2,1,2,b}]
reduce[a_] := a
\end{lstlisting}
To illustrate, the following command yields an empty list:
\begin{lstlisting}[frame=single]
reduce[{2, 3, 2, 3, 2, 3}]
\end{lstlisting}
demonstrating that $(s_2s_3)^3=1$.
Refer to \cite{vanderKallen} for a way to derive such rules for Coxeter groups.

% \begin{lstlisting}[frame=single]
% gamma[1] = {2, 3, 2, 1, 2, 1, 3, 2, 1, 2, 1, 3, 2, 1, 2, 1, 2, 3, 2, 1, 2, 1, 2, 1, 3, 2, 1, 2, 1, 3, 2, 1, 2, 1, 3, 2, 1, 2}; 
% gamma[2] = {1, 2, 3, 2, 1, 2, 1, 2, 1, 3, 2, 1, 2, 1, 3, 2, 1, 2, 1, 2, 3, 2, 1, 2, 1, 2, 1, 3, 2, 1, 2, 1, 3, 2, 1, 2}; 
% gamma[3] = {1, 2, 1, 2, 1, 2, 3, 2, 1, 2, 1, 3, 2, 1, 2, 1, 3, 2, 1, 2, 1, 2, 3, 2, 1, 2, 1, 2, 1, 3, 2, 1, 2, 1, 2, 3, 2, 1, 2, 1, 3, 2,1, 2};
% gamma[4] = {1, 2, 1, 2, 3, 2, 1, 2, 1, 2, 3, 2, 1, 2, 1, 2,  1, 3, 2, 1, 2, 1, 3, 2, 1, 2, 1, 2, 3, 2, 1, 2, 1, 2, 3, 2, 1, 2, 1, 2};
% gamma[5] = {1, 3, 2, 1, 2, 1, 3, 2, 1, 2, 1, 2, 3, 2, 1, 2, 1, 2, 1, 3, 2, 1, 2, 1, 3, 2, 1, 2, 1, 2, 3, 2, 1, 2, 1, 2}; 
% gamma[6] = {2, 3, 2, 1, 2, 1, 2, 3, 2, 1, 2, 1, 2, 1, 3, 2, 1, 2, 1, 3, 2, 1, 2, 1, 2, 3, 2, 1, 2, 1, 2, 1, 3, 2, 1, 2, 1, 2};
% gamma[i_] := Join[{2, 1}, gamma[i - 6], {1, 2}]
% \end{lstlisting}

\section{GAP code for the automorphism group of the Fricke-Macbeath surface}\label{sec:appB}
The GAP code provided below computes the order of the group $G$ in Subsection~\ref{subsec:hurwitzsurfaces} defined by the presentation
$\langle a, b \mid a^{3} = b^{7} = (ab)^{2} = (a^{-1} b^3 a^{-1} b a^{-1} b^3)^2 = 1\rangle$:
\begin{lstlisting}[label=lst:gapcode, frame=single]
f := FreeGroup("a", "b");
a := f.1; b := f.2;
t := a^(-1) * b^3 * a^(-1) * b * a^(-1) * b^3;
G := f / [a^3, b^7, (a*b)^2, t^2];
Size(G);
\end{lstlisting}
The execution of this code yields 504.

Building on the above code, one can obtain the character table of $G$ using the following commands:
\begin{lstlisting}[frame=single]
T:= CharacterTable(G);
Display(T);
\end{lstlisting}

To find the representatives for each of the conjugacy classes, one can use the following commands:
\begin{lstlisting}[frame=single]
conjClasses := ConjugacyClasses(T);

for i in [1..Length(conjClasses)] do
    class := conjClasses[i];
    rep := Representative(class);
    Print(rep, "\n");  
od;
\end{lstlisting}

\bibliographystyle{amsalpha}
\bibliography{main}

\newcommand{\etalchar}[1]{$^{#1}$}
\providecommand{\bysame}{\leavevmode\hbox to3em{\hrulefill}\thinspace}
\providecommand{\MR}{\relax\ifhmode\unskip\space\fi MR }
\providecommand{\MRhref}[2]{%
  \href{http://www.ams.org/mathscinet-getitem?mr=#1}{#2}
}
\providecommand{\href}[2]{#2}
\begin{thebibliography}{CCN{\etalchar{+}}85}

\bibitem[AS89]{MR1028714}
R.~Aurich and F.~Steiner, \emph{Periodic-orbit sum rules for the
  {H}adamard-{G}utzwiller model}, Phys. D \textbf{39} (1989), no.~2-3,
  169--193.

\bibitem[Bea83]{MR0698777}
A.~F. Beardon, \emph{The geometry of discrete groups}, Graduate Texts in
  Mathematics, vol.~91, Springer-Verlag, New York, 1983.

\bibitem[BLS20]{MR4186122}
A.~R. Booker, M.~Lee, and A.~Str\"{o}mbergsson, \emph{Twist-minimal trace
  formulas and the {S}elberg eigenvalue conjecture}, J. Lond. Math. Soc. (2)
  \textbf{102} (2020), no.~3, 1067--1134.

\bibitem[BS07]{MR2338122}
A.~R. Booker and A.~Str\"{o}mbergsson, \emph{Numerical computations with the
  trace formula and the {S}elberg eigenvalue conjecture}, J. Reine Angew. Math.
  \textbf{607} (2007), 113--161.

\bibitem[Bur99]{burnside1899}
W.~Burnside, \emph{Note on the simple group of order 504}, Math. Ann.
  \textbf{52} (1899), 174--176.

\bibitem[Bus10]{MR2742784}
P.~Buser, \emph{Geometry and spectra of compact {R}iemann surfaces}, Modern
  Birkh\"{a}user Classics, Birkh\"{a}user Boston, Ltd., Boston, MA, 2010,
  Reprint of the 1992 edition.

\bibitem[CCN{\etalchar{+}}85]{MR0827219}
J.~H. Conway, R.~T. Curtis, S.~P. Norton, R.~A. Parker, and R.~A. Wilson,
  \emph{{$\Bbb{ATLAS}$} of finite groups}, Oxford University Press, Eynsham,
  1985, Maximal subgroups and ordinary characters for simple groups, With
  computational assistance from J. G. Thackray.

\bibitem[CE03]{MR1973059}
H.~Cohn and N.~Elkies, \emph{New upper bounds on sphere packings. {I}}, Ann. of
  Math. (2) \textbf{157} (2003), no.~2, 689--714.

\bibitem[CM57]{coxeter_moser_1957}
H.~S.~M. Coxeter and W.~O.~J. Moser, \emph{Generators and relations for
  discrete groups}, Ergebnisse der Mathematik und ihrer Grenzgebiete, vol.~14,
  Springer, Berlin, 1957.

\bibitem[Con90]{MR1041434}
M.~Conder, \emph{Hurwitz groups: a brief survey}, Bull. Amer. Math. Soc. (N.S.)
  \textbf{23} (1990), no.~2, 359--370.

\bibitem[Coo18]{cook2018}
J.~Cook, \emph{Properties of eigenvalues on riemann surfaces with large
  symmetry groups}, Ph.D. thesis, Loughborough University, 2018,
  arXiv:2108.11825.

\bibitem[FBP21]{FBP21}
M.~Fortier~Bourque and B.~Petri, \emph{The {K}lein quartic maximizes the
  multiplicity of the first positive eigenvalue of the {L}aplacian}, 2021,
  preprint, arxiv:2111.14699. To appear in Journal of Differential Geometry.

\bibitem[FBP23]{FBP23}
\bysame, \emph{Linear programming bounds for hyperbolic surfaces}, 2023,
  preprint, arXiv:2302.02540.

\bibitem[FK90]{klein1890vorlesungen}
R~Fricke and F.~Klein, \emph{Vorlesungen {\"u}ber die theorie der elliptischen
  modulfunctionen}, vol.~1, B.G. Teubner, 1890.

\bibitem[Fri99]{MR1511059}
R.~Fricke, \emph{Ueber eine einfache {G}ruppe von 504 {O}perationen}, Math.
  Ann. \textbf{52} (1899), no.~2-3, 321--339.

\bibitem[GAP22]{GAP4}
The GAP~Group, \emph{{GAP -- Groups, Algorithms, and Programming, Version
  4.12.2}}, 2022.

\bibitem[Hec12]{MR3043640}
F.~Hecht, \emph{New development in freefem++}, J. Numer. Math. \textbf{20}
  (2012), no.~3-4, 251--265.

\bibitem[Her94]{MR1261121}
S.~M. Hermiller, \emph{Rewriting systems for {C}oxeter groups}, J. Pure Appl.
  Algebra \textbf{92} (1994), no.~2, 137--148.

\bibitem[Hum90]{MR1066460}
J.~E. Humphreys, \emph{Reflection groups and {C}oxeter groups}, Cambridge
  Studies in Advanced Mathematics, vol.~29, Cambridge University Press,
  Cambridge, 1990.

\bibitem[Hur93]{hurwitz1893}
A.~Hurwitz, \emph{Über algebraische gebilde mit eindeutigen transformationen
  in sich}, Mathematische Annalen \textbf{41} (1893), no.~3, 403--442.

\bibitem[Inc20]{Mathematica}
Wolfram~Research{,} Inc., \emph{Mathematica, {V}ersion 12.1}, 2020, Champaign,
  IL.

\bibitem[Jen84]{jenni1984}
F.~Jenni, \emph{Über den ersten {E}igenwert des {L}aplace-{O}perators auf
  ausgewählten {B}eispielen kompakter {R}iemannscher {F}lächen}, Commentarii
  Mathematici Helvetici \textbf{59} (1984), no.~1, 193--203.

\bibitem[JW16]{MR3467692}
G.~A. Jones and J.~Wolfart, \emph{Dessins d'enfants on {R}iemann surfaces},
  Springer Monographs in Mathematics, Springer, Cham, 2016.

\bibitem[Kle79]{Klein79}
F.~Klein, \emph{{\"U}ber die transformationen siebenter ordnung der
  elliptischen funktionen}, Math. Annalen \textbf{14} (1879), 428--471.

\bibitem[KMP21]{kravchuk2023automorphic}
P.~Kravchuk, D.~Mazac, and S.~Pal, \emph{Automorphic spectra and the conformal
  bootstrap}, 2021, preprint, arXiv:2111.12716.

\bibitem[Lee65]{MR0174621}
J.~Leech, \emph{Generators for certain normal subgroups of {$(2,\,3,\,7)$}},
  Proc. Cambridge Philos. Soc. \textbf{61} (1965), 321--332.

\bibitem[Lev99]{MR1722410}
S.~Levy (ed.), \emph{{\it {T}he eightfold way}}, Mathematical Sciences Research
  Institute Publications, vol.~35, Cambridge University Press, Cambridge, 1999.

\bibitem[Mac65]{macbeath1965}
A.M. Macbeath, \emph{On a curve of genus 7}, Proc. London Math. Soc.
  \textbf{15} (1965), no.~9, 527--542.

\bibitem[Mac99]{MR1722414}
A.~M. Macbeath, \emph{Hurwitz groups and surfaces}, The eightfold way, Math.
  Sci. Res. Inst. Publ., vol.~35, Cambridge Univ. Press, Cambridge, 1999,
  pp.~103--113.

\bibitem[Mag74]{MR0352287}
W.~Magnus, \emph{Noneuclidean tesselations and their groups}, Pure and Applied
  Mathematics, vol. Vol. 61, Academic Press, New York-London, 1974.

\bibitem[Mas71]{maskit1971}
B.~Maskit, \emph{On poincare's theorem for fundamental polygons}, Adv. Math.
  \textbf{7} (1971), 219--230.

\bibitem[Ser94]{MR1272022}
J.-P. Serre, \emph{A letter as an appendix to the square-root parameterization
  paper of {A}bhyankar}, Algebraic geometry and its applications ({W}est
  {L}afayette, {IN}, 1990), Springer, New York, 1994, pp.~85--88.

\bibitem[Shi67]{MR0204426}
G.~Shimura, \emph{Construction of class fields and zeta functions of algebraic
  curves}, Ann. of Math. (2) \textbf{85} (1967), 58--159.

\bibitem[SU13]{MR3009726}
A.~Strohmaier and V.~Uski, \emph{An algorithm for the computation of
  eigenvalues, spectral zeta functions and zeta-determinants on hyperbolic
  surfaces}, Comm. Math. Phys. \textbf{317} (2013), no.~3, 827--869.

\bibitem[{The}22]{PARI2}
{The PARI~Group}, Univ. Bordeaux, \emph{{PARI/GP version \texttt{2.15.2}}},
  2022, available from \url{http://pari.math.u-bordeaux.fr/}.

\bibitem[{The}23]{sagemath}
{The Sage Developers}, \emph{{S}agemath, the {S}age {M}athematics {S}oftware
  {S}ystem ({V}ersion 10.1)}, 2023, {\tt https://www.sagemath.org}.

\bibitem[vdK02]{vanderKallen}
W.~van~der Kallen, \emph{Computing some {KL}-polynomials for the poset of {$B
  \times B$}-orbits in group compactifications}, preprint, 2002.

\bibitem[Vog03]{Vogeler2003}
R.~Vogeler, \emph{On the {G}eometry of {H}urwitz {S}urfaces}, Ph.D. thesis,
  Florida State University, 2003.

\bibitem[Voi09]{MR2541438}
J.~Voight, \emph{Computing fundamental domains for {F}uchsian groups}, J.
  Th\'{e}or. Nombres Bordeaux \textbf{21} (2009), no.~2, 469--491.

\end{thebibliography}
\end{document}